\documentclass[11pt]{article}

\usepackage[margin=1in]{geometry}
\usepackage[numbers,sort&compress]{natbib}

\usepackage[utf8]{inputenc}
\usepackage[T1]{fontenc}
\usepackage{url}
\usepackage{amsmath,amsthm,amssymb,amsfonts}
\usepackage{enumerate}
\usepackage{subfigure}
\usepackage{multirow}
\usepackage{hhline}
\usepackage{xcolor}
\usepackage{tikz}
\usetikzlibrary{arrows.meta,positioning,calc,patterns}
\usepackage{mdframed}
\usepackage{hyperref}
\usepackage{graphicx}
\usepackage{graphics}
\usepackage{algorithm}
\usepackage{algpseudocode}
\usepackage{epsfig,psfrag}
\usepackage{tabularx}
\usepackage{tabulary}
\usepackage{booktabs}
\usepackage{nicefrac}
\usepackage{microtype}
\usepackage[noabbrev,capitalise,nameinlink]{cleveref}

\hypersetup{
  breaklinks=true,
  colorlinks=true,
  linkcolor=blue,
  citecolor=blue,
  urlcolor=blue
}

\DeclareMathOperator{\argmax}{arg\,max}
\DeclareMathOperator{\Argmax}{Arg\,max}

\DeclareMathOperator{\co}{co}

\DeclareMathOperator{\dist}{dist}

\newtheorem{theorem}{Theorem}
\newtheorem{example}{Example}

\newtheorem{corollary}{Corollary}

\newtheorem{lemma}{Lemma}
\newtheorem*{lemma*}{Lemma}
\newtheorem{definition}{Definition}
\newtheorem{proposition}{Proposition}

\crefname{assumption}{Assumption}{Assumptions}
\Crefname{assumption}{Assumption}{Assumptions}

\newmdenv[
  linewidth=0.6pt,
  linecolor=black,
  skipabove=6pt,
  skipbelow=6pt,
  innertopmargin=6pt,
  innerbottommargin=6pt,
  innerleftmargin=8pt,
  innerrightmargin=8pt
]{restatementbox}

\newcommand{\R}{\mathbb R}
\newcommand{\one}{\mathbf 1}
\newcommand{\proj}{\boldsymbol{\Pi}_\perp}
\newcommand{\calS}{\mathcal S}
\newcommand{\calA}{\mathcal A}
\newcommand{\calH}{\mathcal H}

\newcommand{\missingfigure}[1]{%
  \fbox{%
    \begin{minipage}{0.68\linewidth}
    \centering
    Missing figure file: \texttt{#1}
    \end{minipage}%
  }%
}
\newcommand{\safeincludegraphics}[2][]{%
  \IfFileExists{#2}{\includegraphics[#1]{#2}}{%
    \IfFileExists{#2.pdf}{\includegraphics[#1]{#2}}{%
      \IfFileExists{#2.png}{\includegraphics[#1]{#2}}{%
        \IfFileExists{#2.eps}{\includegraphics[#1]{#2}}{\missingfigure{#2}}%
      }%
    }%
  }%
}

\allowdisplaybreaks

\title{Switching-Geometry Analysis of Deflated Q-Value Iteration}

\author{%
Donghwan Lee\\
Department of Electrical Engineering\\
Korea Advanced Institute of Science and Technology (KAIST)\\
Daejeon 34141, South Korea\\
\texttt{donghwan@kaist.ac.kr}
}

\date{}

\begin{document}

\maketitle

\begin{abstract}
This paper develops a joint spectral radius (JSR) framework for analyzing rank-one deflated Q-value iteration (Q-VI) in discounted Markov decision process control. Focusing on an all-ones residual correction, we interpret the resulting algorithm through the geometry of switching systems and, to the best of our knowledge, give the first JSR-based convergence analysis of deflated Q-VI for policy optimization problems.
Our analysis reveals that the standard Q-VI switching system model has JSR exactly the discount factor $\gamma\in (0,1)$, since all admissible subsystems share the all-ones vector as an invariant direction. By passing to the quotient space that removes this direction, we obtain a projected switching system model whose JSR governs the relevant error dynamics and may be strictly smaller than $\gamma$. Therefore, the deflated Q-VI admits a potentially sharper convergence-rate characterization than the ambient-space $\gamma$-bound. Finally, we prove that the correction is equivalent to a scalar recentering of standard Q-VI. Hence, the projected trajectory, and therefore the greedy-policy sequence, is unchanged relative to standard Q-VI initialized from the same point. The benefit of deflation is not a change in the induced decision-making problem, but a more precise JSR-based description of the convergence geometry after the redundant all-ones component is removed.
\end{abstract}

\section{Introduction}
Q-value iteration (Q-VI) is a foundational dynamic-programming algorithm for
finite discounted Markov decision processes (MDPs)
\cite{bertsekas1996neuro,puterman2014mdp}. Its standard deterministic
convergence analysis relies on the fact that the Bellman optimality operator is
a $\gamma$-contraction in the infinity norm
\cite{bertsekas1996neuro,puterman2014mdp}. Consequently, Q-VI converges to the
unique optimal action-value function $Q^\star$ at the classical exponential
rate $\gamma \in (0,1)$.

Although this contraction argument is fundamental, it is geometrically
conservative. A switching system viewpoint~\cite{lee2026beyond} makes this
conservatism explicit. The linear switching family associated with standard
Q-VI~\cite{liberzon2003switching,hespanha1999stability,shorten2007stability,lin2009stability}
has joint spectral radius (JSR)~\cite{tsitsiklis1997lyapunov,rota1960jsr,blondel2005jsr,jungers2009jsr}
exactly equal to $\gamma$, because every subsystem preserves the all-ones
direction and acts on it with eigenvalue $\gamma$. After quotienting out this
invariant direction, the transverse dynamics are governed by a projected
switching family with the so-called projected JSR $\bar\rho$, which can be strictly smaller
than the ambient contraction rate $\gamma$. Nevertheless, standard Q-VI can
still converge to $Q^\star$ no faster than $\gamma$ in the worst case, since a
pure all-ones error evolves exactly as $\gamma^k \one$.

In this paper, we study the (rank-one) deflated Q-VI update
\[
Q_{k+1}
=
F(Q_k)
+
\frac{\gamma}{1-\gamma}\,
d^\top \bigl(F(Q_k)-Q_k\bigr)\one,
\]
where $F$ is the Bellman optimality operator, $Q_k$ is the current estimate of
$Q^\star$, and $d$ is either the uniform distribution or a fixed
state-action distribution. Closely related deflated value-iteration schemes
were proposed in the classical work of~\cite{bertsekas1995rankone} and in the
recent rank-one modified value-iteration framework of~\cite{kolarijani2025rankone}.

The main contribution of the present paper is to provide, to the best of our
knowledge, the first rigorous JSR-based convergence theory for
deflated Q-VI in discounted MDP control problems. Our analysis is built on an exact
switching system formulation of the deflated Q-VI recursion~\cite{liberzon2003switching,hespanha1999stability,shorten2007stability,lin2009stability}
combined with JSR theory~\cite{tsitsiklis1997lyapunov,rota1960jsr,blondel2005jsr,jungers2009jsr}.
Beyond proving convergence, this framework clarifies the precise geometric
mechanism behind the improved rate. Standard Q-VI error can be decomposed into
a scalar coordinate in the all-ones direction $\one$ and a quotient, or
transverse, component obtained after identifying Q-functions that differ only by
a uniform shift $c\one$. The scalar all-ones coordinate is the reason the ambient
switching family has JSR exactly $\gamma$: every subsystem maps this direction
to itself with eigenvalue $\gamma$. The rank-one residual correction cancels the
autonomous evolution of this scalar coordinate. At the same time, it does not
change the quotient trajectory, because the deflated iterate remains equal to
the corresponding standard Q-VI iterate plus a scalar multiple of $\one$. Since
such a uniform shift does not change statewise action comparisons, the same
greedy-policy sequence is preserved. In this sense, the proposed analysis gives
a sharper geometric understanding of deflated Q-VI than the classical
norm-contraction argument.

The main contributions are summarized as follows.
\begin{enumerate}[(1)]
\item We derive exact switching-system error representations for both standard
Q-VI and the proposed deflated Q-VI. These representations separate the dynamics
along the all-ones direction $\one$ from the transverse, policy-dependent
dynamics on the quotient space.

\item We prove that the rank-one residual correction in Q-VI eliminates the
autonomous scalar dynamics in the all-ones direction. As a result, the switching
system governing deflated Q-VI has JSR equal to the projected JSR $\bar\rho$.
Hence, the relevant convergence rate of the deflated Q-VI dynamics is
characterized by $\bar\rho$, which is no larger than $\gamma$ and may be
strictly smaller. We also show that the uniform residual average can be replaced
by any fixed state-action distributional average: for every fixed admissible
vector $d$, the same quotient JSR is preserved through an oblique-projection
similarity. This fixed-$d$ result is distinct from the time-varying distribution
estimates used in adaptive rank-one or Q-learning schemes~\cite{kolarijani2025rankone}.

\item We identify the structural meaning of the rank-one correction. The
deflated Q-VI iterate differs from the corresponding standard Q-VI iterate only
by a scalar multiple of $\one$. Because the all-ones direction does not affect
the policy ordering of greedy actions, the projected trajectory is unchanged,
and the greedy-policy sequence is preserved relative to standard Q-VI
initialized from the same point. Thus, the rank-one correction should be
interpreted as accelerating Q-function error convergence through re-centering,
rather than as accelerating policy identification; see
Figure~\ref{fig:deflated-qvi-geometry}.
\end{enumerate}

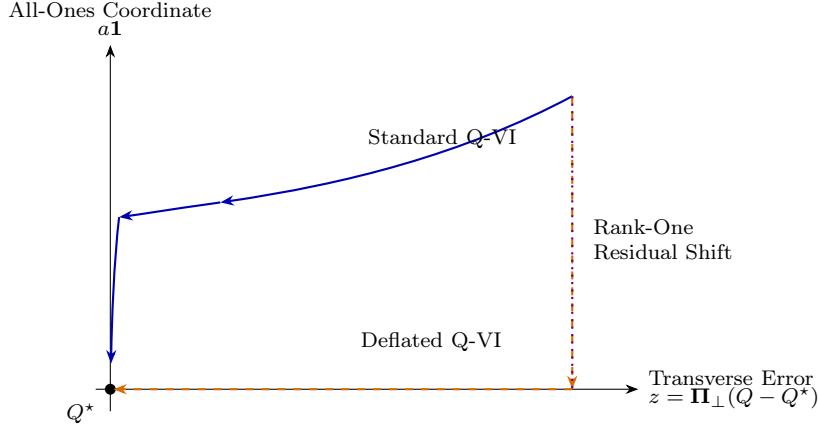
\begin{figure}[t]
\centering
\begin{tikzpicture}[
    >=Stealth,
    scale=0.97,
    every node/.style={font=\scriptsize},
    ordinarytraj/.style={thick,blue!70!black,-{Stealth[length=1.8mm]}},
    defltraj/.style={thick,dashed,orange!85!black,-{Stealth[length=1.8mm]}},
    shiftline/.style={dotted,thick,violet!80!black},
    point/.style={circle,fill=black,inner sep=1.45pt}
]
    \draw[->] (-0.2,0) -- (7.2,0)
        node[right,align=left] {Transverse Error\\[-1mm] \(z=\proj(Q-Q^\star)\)};
    \draw[->] (0,-0.3) -- (0,4.7)
        node[above,align=center] {All-Ones Coordinate\\[-1mm] \(a\one\)};

    \node[point,label=below left:{\(Q^\star\)}] at (0,0) {};

    \draw[ordinarytraj] (6.3,4.0) .. controls (4.8,3.2) and (3.2,2.8) .. (1.5,2.55);
    \draw[ordinarytraj] (1.5,2.55) .. controls (0.78,2.45) and (0.36,2.38) .. (0.12,2.35);
    \draw[ordinarytraj] (0.12,2.35) .. controls (0.05,1.75) and (0.02,0.95) .. (0.01,0.35);
    \node[anchor=west,align=left,text width=3.0cm] at (3.38,3.42)
        {Standard Q-VI};

    \draw[defltraj] (6.3,4.0) -- (6.3,0);
    \draw[defltraj] (6.3,0) .. controls (4.6,0) and (2.4,0) .. (0.04,0);
    \node[anchor=west,align=left,text width=3.1cm] at (3.28,0.65)
        {Deflated Q-VI};

    \draw[shiftline] (6.3,4.0) -- (6.3,0);
    \node[anchor=west,align=left,text width=2.0cm] at (6.45,2.05)
        {Rank-One Residual Shift};
\end{tikzpicture}
\caption{Geometric mechanism of the deflated Q-VI analysis, following
the projected Q-VI geometry of Lee~\cite{lee2026beyond}. Standard Q-VI may
retain a slowly decaying all-ones component. The residual correction preserves
the projected trajectory and re-centers the iterate in the all-ones direction,
so the full switched-family rate is governed by the projected rate
\(\bar\rho\).}
\label{fig:deflated-qvi-geometry}
\end{figure}

\section{Related Work}
\label{sec:related-work}

Rank-one residual corrections for value iteration were studied in
\cite{bertsekas1995rankone} for linear fixed-point iterations. In the
discounted Markovian setting, the all-ones residual correction has the same
basic recentering structure as the deflated Q-VI considered in this paper. More
recently, the rank-one modified value iteration (R1-VI) framework in~\cite{kolarijani2025rankone} used rank-one approximations of transition
dynamics to develop planning and learning algorithms, including Q-function
variants, with convergence guarantees comparable to those of standard VI and
Q-learning. These works provide important algorithmic and approximation
perspectives on rank-one corrections.

Related deflation ideas have also appeared in spectral modifications of value
iteration. For example, deflated Dynamics value iteration~\cite{lee2025ddvi} removes dominant eigenspaces of a transition matrix and obtains convergence rates determined by the remaining spectral components. The closest geometric predecessor to the present work is~\cite{lee2026beyond}, which analyzes standard Q-VI through projected switching dynamics and shows that the corresponding rate governs the action-ordering-relevant component, and therefore policy identification.

The present paper is distinct from these works in both scope and analysis. We analyze the exact switching-system error dynamics of deflated Q-VI in the policy optimization setting. To the best of our knowledge, this provides the first rigorous JSR-based convergence proof for deflated Q-VI in
discounted MDP control. The key conclusion is that the JSR of the
switching system model of deflated Q-VI is exactly the projected JSR
$\bar\rho$, which can be strictly smaller than the classical contraction rate
$\gamma$. At the same time, the deflated Q-VI iterate differs from standard Q-VI
only by an all-ones shift, so the projected trajectory and greedy-policy
sequence are unchanged.

\section{Preliminaries}\label{sec:preliminaries}
\subsection{Notation}
For a finite set \(\mathcal Y\), its cardinality is denoted by \(|\mathcal Y|\).
The symbols \(\R\), \(\R^n\), and \(\R^{n\times m}\) denote the real numbers, the
\(n\)-dimensional Euclidean space, and the set of \(n\times m\) real matrices,
respectively. All vectors are column vectors. For a matrix \(A\), \(A^\top\)
denotes its transpose, and \(\ker(A):=\{x:Ax=0\}\) denotes its nullspace. The
identity matrix is denoted by \(I\), and the vector with all entries equal to
one is denoted by \(\one\). For a linear map \(B\) and a subspace \(W\) such
that \(B(W)\subseteq W\), the notation \(B|_W\) denotes the restriction of \(B\)
to \(W\), namely \(B|_W:W\to W\). For \(m\ge1\), the probability simplex is $\Delta_m:=\left\{p\in\R^m:p_i\ge0,\ \sum_{i=1}^m p_i=1\right\}$.
For a convex set \(\mathcal Y\subset\R^n\) and a vector \(x\in\R^n\), define $\dist_2(x,\mathcal Y):=\inf_{y\in\mathcal Y}\|x-y\|_2$ and $\dist_\infty(x,\mathcal Y):=\inf_{y\in\mathcal Y}\|x-y\|_\infty$.
For \(x\in\R^{|\calS||\calA|}\), its empirical mean and its orthogonal
projection onto \(\operatorname{span}(\one)^\perp\) are defined by
\[
\bar x:=\frac1{|\calS||\calA|}\one^\top x,
\qquad
\proj x:=\left(I-\frac1{|\calS||\calA|}\one\one^\top\right)x.
\]
The same symbol is used for the projection matrix
\[
\proj:=I-\frac1{|\calS||\calA|}\one\one^\top,
\]
which satisfies
\begin{equation}
\label{eq:projection-kills-one}
\proj\one=0.
\end{equation}
The set-valued maximizer is denoted by \(\Argmax\), while \(\argmax\) denotes a
fixed tie-broken single-valued maximizer.

\subsection{Finite discounted Markov decision process}
Let us consider a finite discounted Markov decision process (MDP)~\cite{puterman2014mdp,bertsekas2015dynamic} with finite state space
\(\calS=\{1,\ldots,|\calS|\}\), finite action space
\(\calA=\{1,\ldots,|\calA|\}\), transition probability
\(P(s'\mid s,a)\), reward \(r(s,a,s')\), and discount factor
\(\gamma\in(0,1)\). The expected reward associated with a state-action pair is
$R(s,a):=\sum_{s'\in\calS}P(s'\mid s,a)r(s,a,s')$, for
$(s,a)\in\calS\times\calA$.

A deterministic stationary policy is a map
\(\pi:\calS\to\calA\). More generally, a stochastic stationary policy is a map
\(\mu:\calS\to\Delta_{|\calA|}\), where \(\mu(a\mid s)\) denotes the
probability of selecting action \(a\) in state \(s\). For a deterministic
stationary policy \(\pi\), the associated action-value function is
\[
Q^\pi(s,a)
:=
\mathbb E\left[
\left.
\sum_{t=0}^{\infty}\gamma^t r(s_t,a_t,s_{t+1})
\right|s_0=s,\ a_0=a,\ a_t=\pi(s_t)\ \forall t\ge1
\right],
\]
where the initial action is fixed to be \(a\), and subsequent actions are chosen
according to \(\pi\).

The optimal Q-function and the optimal value function are defined by
\[
Q^\star(s,a):=\sup_\pi Q^\pi(s,a),
\qquad
V^\star(s):=\max_{a\in\calA}Q^\star(s,a).
\]
Equivalently, \(Q^\star\) is the unique fixed point of the Bellman optimality
equation
\[
Q^\star(s,a)
=
R(s,a)
+
\gamma
\sum_{s'\in\calS}P(s'\mid s,a)
\max_{a'\in\calA}Q^\star(s',a').
\]
For each state \(s\), the set of optimal greedy actions is denoted by
\[
\Phi^\star(s)
:=
\Argmax_{a\in\calA}Q^\star(s,a).
\]
The set of all deterministic policies is denoted by $\Theta:=\{\pi:\calS\to\calA\}$.
Throughout the paper,
we use a fixed tie-breaking rule and denote by
\[
\pi_Q(s)
:=
\argmax_{a\in\calA}Q(s,a),
\qquad s\in\calS,
\]
the tie-broken greedy policy induced by \(Q\). Thus \(\pi_Q\in\Theta\) is
well-defined for every \(Q\).

\subsection{Q-value iteration}

The Bellman optimality operator \(F:\R^{|\calS||\calA|}\to\R^{|\calS||\calA|}\)
is defined componentwise by
\[
(FQ)(s,a)
:=
R(s,a)+\gamma\sum_{s'\in\calS}P(s'\mid s,a)
\max_{a'\in\calA}Q(s',a')\quad \forall (s,a) \in {\cal S}\times {\cal A}.
\]
The standard Q-value iteration (Q-VI) is the deterministic recursion
\[
Q_{k+1}=F(Q_k), \qquad k\in\{0,1,2,\ldots\}.
\]
The Bellman operator is a \(\gamma\)-contraction in the infinity
norm~\cite{bertsekas1996neuro,puterman2014mdp}:
\[
\|F(Q)-F(\widetilde Q)\|_\infty
\le
\gamma\|Q-\widetilde Q\|_\infty,
\qquad
\forall Q,\widetilde Q \in {\mathbb R}^{|{\cal S}||{\cal A}|}.
\]
Therefore, $\|Q_k-Q^\star\|_\infty \le \gamma^k\|Q_0-Q^\star\|_\infty$.
The Bellman operator also satisfies the shift identity
\begin{equation}
\label{eq:shift-identity}
F(Q+c\one)=F(Q)+\gamma c\one,
\qquad
\forall Q,
\quad
\forall c\in\R.
\end{equation}
This identity is the source of the all-ones mode, $\one$, in standard Q-VI. If the
current iterate is shifted by a uniform amount, $c\one$, then the Bellman backup maps this shift to the uniform shift $\gamma c\one$. In particular, if \(Q=Q^\star+c\one\), then
\[
F(Q)-Q^\star=\gamma c\one.
\]
Consequently, a pure error of the form \(c\one\) is mapped to \(\gamma c\one\), independently of rewards,
transition probabilities, or the active greedy policy. In the switching system
representation below, this $\one$ becomes a common invariant direction of every
subsystem of the switching system. Since this direction, $\one$, changes all state-action values by the same amount, it does not affect greedy action preferences, but it still appears in the full Q-function error and forces the standard worst-case full-error rate to include the factor \(\gamma\).

The next section rewrites Q-VI as a switching system. This form makes
it possible to separate the all-ones direction, $\one$, from the policy-relevant
transverse directions. To keep the main exposition focused and to make the
logical flow easier to follow, all proofs are deferred to the appendix.

\section{Switching system model of Q-VI}
\label{sec:switching-model}
This section rewrites Q-VI as an affine switching system~\cite{liberzon2003switching,hespanha1999stability,shorten2007stability,lin2009stability}. The affine form keeps
the reward and greedy policy dependent offset explicit, while the homogeneous
linear part is the object whose joint spectral radius determines the worst-case
switched rate.

\subsection{Switching systems and joint spectral radius}

An affine switching system has the form~\cite{liberzon2003switching,hespanha1999stability,shorten2007stability,lin2009stability}
\[
  x_{k+1}=A_{\sigma_k}x_k+b_{\sigma_k},
  \qquad
  \sigma_k\in\mathcal M:=\{1,\ldots,M\},
  \qquad
  k\in\{0,1,2,\ldots\},
\]
where \(\sigma_k\) is a switching signal, \(A_{\sigma_k}\) is the active subsystem selected from the switching family \(\mathcal H:=\{A_1,A_2,\ldots,A_M\}\), and \(b_{\sigma_k}\) is a
mode-dependent affine term. When \(b_{\sigma_k}=0\), the homogeneous part
reduces to a switched linear system,
\[
  x_{k+1}=A_{\sigma_k}x_k,
  \qquad
  k\in\{0,1,2,\ldots\}.
\]
The worst-case exponential rate of the switched linear family $\mathcal H$ is characterized
by the joint spectral radius (JSR)~\cite{tsitsiklis1997lyapunov,rota1960jsr,blondel2005jsr,jungers2009jsr},
defined as follows.
\begin{definition}[Joint spectral radius]\label{def:jsr}
For a bounded set of matrices \({\calH}\subset\R^{m\times m}\), its joint
spectral radius is
\[
\rho({\calH})
:=
\lim_{k\to\infty}
\sup_{A_1,\ldots,A_k\in {\calH}}
\|A_k\cdots A_1\|^{1/k},
\]
where \(\|\cdot\|\) denotes any fixed submultiplicative matrix norm. The value is
independent of the chosen submultiplicative norm~\cite{rota1960jsr,jungers2009jsr}.
\end{definition}

When \({\calH}\) is finite, the supremum for each fixed product length is a
maximum over products generated by matrices in \({\calH}\). Thus, for fixed
\(k\), one may write \(\max_{A_1,\ldots,A_k\in\calH}\) instead of
\(\sup_{A_1,\ldots,A_k\in\calH}\). For a finite family \(\calH\), the notation
\(\rho(\co({\calH}))\) means the JSR computed when each factor in a product is
allowed to be any convex combination of matrices in \({\calH}\).
The next three lemmas collect standard JSR facts used throughout the switching
analysis.
\begin{lemma}[Convexification invariance of the JSR~{\cite[Prop.~1.8]{jungers2009jsr}}]
\label{lem:jsr-convexification}
For every finite matrix family \(\calH\), the following identity holds:
\[
\rho(\co(\calH))=\rho(\calH).
\]
\end{lemma}

\noindent This result appears in~\cite[Prop.~1.8]{jungers2009jsr}.
\begin{lemma}[Uniform exponential product bound]
\label{lem:jsr-product-bound}
Let \(\|\cdot\|\) be any fixed submultiplicative matrix norm. If
\(\rho({\calH})<1\), then every switched product is uniformly exponentially
stable in that norm: for every \(\varepsilon\in(0,1-\rho(\calH))\), with
\[
\beta_\varepsilon:=\rho(\calH)+\varepsilon,
\]
there exists \(C_{\beta_\varepsilon}>0\), depending on
\(\beta_\varepsilon\) and on the chosen norm, such that for every \(k\ge0\) and
every switching sequence \(\sigma_0,\ldots,\sigma_{k-1} \in {\mathcal M}\) with
\[
A_j:=A_{\sigma_j}\in\calH,
\qquad j=0,\ldots,k-1,
\]
one has
\[
\left\|A_{k-1}\cdots A_0\right\|
\le
C_{\beta_\varepsilon}\beta_\varepsilon^k,
\]
where the product is interpreted as the identity matrix $I$ when \(k=0\).
\end{lemma}
The following lemma is given in~\cite[Eq.~(34)]{cicone2015note}.
\begin{lemma}[Block upper-triangular JSR~{\cite[Eq.~(34)]{cicone2015note}}]
\label{lem:block-upper-triangular-jsr}
Let
\[
\mathcal M:=
\left\{
\begin{bmatrix}
B_i & C_i\\
0 & D_i
\end{bmatrix}:i\in\mathcal I
\right\}
\]
be a bounded family of block upper-triangular matrices. Let
\(\mathcal B:=\{B_i:i\in\mathcal I\}\) and
\(\mathcal D:=\{D_i:i\in\mathcal I\}\). Then
\[
\rho(\mathcal M)=\max\{\rho(\mathcal B),\rho(\mathcal D)\}.
\]
\end{lemma}

The preceding lemmas provide the technical JSR toolkit. The next subsection
returns to the Q-VI model and writes its Bellman update in vectorized
switched-system form.

\subsection{Vectorized representation}

Each Q-function is identified with a vector \(Q\in\R^{|\calS||\calA|}\) whose
entries enumerate \(Q(s,a)\). The vectorized state-action transition notation,
the use of stochastic policies to represent the Bellman maximization error, and
the projected switching viewpoint used in this subsection are adapted
from~\cite{lee2026qlearning} and~\cite{lee2026beyond}. We use the compact notation
\[
P:=
\begin{bmatrix}
P_1\\ \vdots\\ P_{|\calA|}
\end{bmatrix}
\in\R^{(|\calS|\,|\calA|)\times |\calS|},
\qquad
R:=
\begin{bmatrix}
R(\cdot,1)\\ \vdots\\ R(\cdot,|\calA|)
\end{bmatrix}
\in\R^{|\calS|\,|\calA|},
\]
where \(P_a=P(\cdot\mid\cdot,a)\in\R^{|\calS|\times |\calS|}\),
\(a\in\calA\). For any stochastic policy
\(\mu:\calS\to\Delta_{|\calA|}\), define
\[
\Pi^\mu
:=
\begin{bmatrix}
\mu(1)^\top\otimes e_1^\top\\
\mu(2)^\top\otimes e_2^\top\\
\vdots\\
\mu(|\calS|)^\top\otimes e_{|\calS|}^\top
\end{bmatrix}
\in\R^{|\calS|\times(|\calS|\,|\calA|)}.
\]
Then \(P\Pi^\mu\in\R^{(|\calS|\,|\calA|)\times(|\calS|\,|\calA|)}\) is the
state-action transition matrix under \(\mu\). For a deterministic policy
\(\pi\in\Theta\), the same notation \(\Pi^\pi\) is used by identifying
\(\pi(s)\) with its one-hot encoding for each \(s\in\calS\). Thus
\(\Pi^\pi\) is obtained by using the one-hot vector corresponding to
\(\pi(s)\) in each row. For \(Q\in\R^{|\calS||\calA|}\), let us define
\[
\Pi_Q:=\Pi^{\pi_Q},
\]
where $\pi_Q$ is the tie-broken greedy policy induced by \(Q\)
\[
\pi_Q(s):=\argmax_{a\in\calA}Q(s,a),
\qquad s\in\calS.
\]
With this notation, the Bellman operator can be written as
\[
F(Q)=R+\gamma P\Pi_Q Q.
\]

\subsection{Affine switching-system representation}

For each deterministic policy \(\pi\in\Theta\), define the matrix
\[
A_\pi:=\gamma P\Pi^\pi
\in {\mathbb R}^{|{\cal S}||{\cal A}|\times |{\cal S}||{\cal A}|}.
\]
Then Q-VI can be written as the affine switching system
\[
Q_{k+1}=R+A_{\pi_{Q_k}}Q_k,
\qquad
k\in\{0,1,2,\ldots\},
\]
where the affine term is the reward vector \(R\), and the active subsystem is
selected by the tie-broken greedy policy of the current iterate. In error
coordinates \(e_k:=Q_k-Q^\star\), the same recursion becomes
\[
e_{k+1}
=
A_{\pi_{Q_k}}e_k
+
\gamma P(\Pi^{\pi_{Q_k}}-\Pi^{\pi_{Q^\star}})Q^\star,
\qquad
k\in\{0,1,2,\ldots\}.
\]
where \(\pi_{Q^\star}\) is any tie-broken greedy policy at \(Q^\star\). The
linear part of the full deterministic switching family is
\[
\calH:=\{A_\pi:\pi\in\Theta\}.
\]
The following proposition establishes the exact JSR of this full linear family.
\begin{proposition}[Full Q-VI switching JSR]
\label{prop:qvi-full-jsr}
The deterministic switching family \(\calH=\{A_\pi:\pi\in\Theta\}\) satisfies
\[
\rho(\calH)=\gamma.
\]
\end{proposition}
Standard Q-VI contains an unavoidable all-ones mode with
eigenvalue \(\gamma\). This observation explains why the usual full-error
convergence bound can be conservative for policy identification and motivates
the projected representation and residual recentering studied below.

\section{Exact switching system}
\label{sec:exact-switching-system}
This section gives the exact stochastic-policy switching representation of the
Bellman optimality error dynamics. Let
\[
\mathcal M_{\rm st}:=\{\mu:\calS\to\Delta_{|\calA|}\}
\]
be the set of stochastic stationary policies. For a stochastic policy
\(\mu\in\mathcal M_{\rm st}\), define the matrix
\[
A_\mu:=\gamma P\Pi^\mu,
\]
which satisfies
\begin{equation}
\label{eq:common-ones-eigenrelation}
A_\mu\one=\gamma\one,
\qquad
\forall \mu\in\mathcal M_{\rm st}.
\end{equation}
This identity follows directly from the stochasticity of the policy and the
transition kernel. The matrices \(A_\mu\) will be the subsystems of the
stochastic-policy switching representation introduced below. Thus,
\cref{eq:common-ones-eigenrelation} implies that the all-ones direction is a
common invariant direction of every stochastic-policy subsystem, with eigenvalue
\(\gamma\).

The following representation uses the stochastic-policy linearization of the
Bellman max that appears in~\cite{lee2026qlearning} and~\cite[Lemma~5]{lee2026beyond}.
The broader viewpoint of analyzing value iteration through affine or switching
dynamical systems is also close to the first-order approach of~\cite{goyal2023firstorder}.
\begin{lemma}[Exact stochastic-policy error representation]
\label{lem:exact-representation}
For every \(Q\in\R^{|\calS||\calA|}\), there exists a stochastic policy
\(\mu_Q:\calS\to\Delta_{|\calA|}\) such that
\[
F(Q)-Q^\star
=
A_{\mu_Q}(Q-Q^\star).
\]
Consequently, with \(e:=Q-Q^\star\),
\[
F(Q)-Q=(A_{\mu_Q}-I)e.
\]
In particular, the statement applies to standard Q-VI iterates and to the
deflated Q-VI iterates studied below.
\end{lemma}
For standard Q-VI, let \(Q_{k+1}=F(Q_k)\),
\(e_k:=Q_k-Q^\star\), and let \(\mu_k:=\mu_{Q_k}\) be a stochastic
policy supplied by~\cref{lem:exact-representation}. Then,~\cref{lem:exact-representation} implies that standard Q-VI admits
the homogeneous stochastic-policy switching system representation
\begin{equation}
\label{eq:qvi-stochastic-switching-system}
e_{k+1}=A_{\mu_k}e_k,
\qquad
A_{\mu_k}\in\calH_{\rm st}
:=\{A_\mu:\mu:\calS\to\Delta_{|\calA|}\},
\qquad
k\in\{0,1,2,\ldots\}.
\end{equation}
Equivalently, \(Q_{k+1}=Q^\star+A_{\mu_k}(Q_k-Q^\star)\). The switching
signal is selected by the Bellman max-linearization at the current iterate; that
is, the realized nonlinear trajectory uses \(\mu_k=\mu_{Q_k}\). Thus
\cref{eq:qvi-stochastic-switching-system} is an exact trajectory-wise
representation, while \(\calH_{\rm st}\) is the associated enlarged
stochastic-policy family used for switched-family rate bounds. Its deterministic
subfamily, namely the subset of $\calH_{\rm st}$ corresponding to all deterministic policies, is \(\calH=\{A_\pi:\pi\in\Theta\}\).

\section{Projected switching system for standard Q-VI}
\label{sec:projected-switching}

In this section, we introduce the projected switching system associated with
the Q-VI error dynamics. Starting from the switching system representation
developed in the previous section, we project the dynamics onto the zero-sum transverse subspace
\[
\operatorname{span}(\one)^\perp
=
\{x\in\R^{|\calS||\calA|}:\one^\top x=0\}.
\]
This projection removes the all-ones
component and isolates the transverse dynamics. The resulting projected
switching system captures the part of the Q-function error that is relevant to
greedy action preferences, and it will play a central role in interpreting the
convergence behavior of deflated Q-VI.

The reason for focusing on \(\operatorname{span}(\one)^\perp\) is geometric.
Adding a scalar multiple \(c\one\) to a Q-function shifts all state-action
values by the same constant. Therefore, for each fixed state, all action values
are shifted equally, and neither the statewise action ordering nor the
tie-broken greedy action is changed. Consequently, the component of the error
\(e_k\) along the all-ones direction affects only the absolute numerical level
of the Q-function, but not the policy selected from it. Therefore, the action-ordering-relevant part of the error is the residual component
obtained after removing the uniform all-ones shift. Equivalently, the projected
error on \(\operatorname{span}(\one)^\perp\) represents the transverse part of
the dynamics that can influence greedy decisions. This viewpoint provides the
geometric foundation for the JSR-based convergence analysis of deflated Q-VI
developed below.
To track this action-ordering-relevant part, define the projected subsystem
\[
\bar A_\mu:=\proj A_\mu\proj.
\]
The projection before \(A_\mu\) chooses the zero-mean representative of the current error, and the projection after \(A_\mu\) removes
any new uniform shift generated by the subsystem. Therefore, \(\bar A_\mu\) keeps
exactly the transverse dynamics that matter for greedy action comparisons and
discards the uninformative all-ones coordinate.
Formally, the associated projected switching system is
\begin{equation}
\label{eq:projected-switching-system}
z_{k+1}=\bar A_{\mu_k}z_k,
\qquad
z_k\in\operatorname{span}(\one)^\perp,
\qquad
\bar A_{\mu_k}\in\bar{\calH}_{\rm st}
:=\{\bar A_\mu:\mu:\calS\to\Delta_{|\calA|}\},
\qquad
k\in\{0,1,2,\ldots\}.
\end{equation}
Every element of \(\bar{\calH}_{\rm st}\) is regarded as a linear map on
\(\operatorname{span}(\one)^\perp\). The corresponding projected family with policies restricted to deterministic policies is
\[
\bar{\calH}:=\{\bar A_\pi:\pi\in\Theta\}.
\]
\begin{definition}[Projected joint spectral radius]\label{def:projected-jsr}
The projected JSR of the deterministic projected family \(\bar{\calH}\) is
\[
\bar\rho
:=\rho(\bar{\calH})
:=
\lim_{k\to\infty}
\max_{\pi_0,\ldots,\pi_{k-1}\in\Theta}
\left\|
\bar A_{\pi_{k-1}}
\cdots
\bar A_{\pi_0}
\right\|^{1/k}.
\]
\end{definition}
Equivalently, \(\bar\rho\) is the
worst-case exponential rate of products of the projected matrices.

\begin{lemma}[Empirical orthogonal error decomposition]
\label{lem:empirical-orthogonal-error-decomposition}
For any error vector
\(e_k\in\R^{|\calS||\calA|}\), define \(z_k\) and \(a_k\) as
\begin{equation}
\label{eq:empirical-error-decomposition}
z_k:=\proj e_k,
\qquad
a_k:=\frac1{|\calS||\calA|}\one^\top e_k.
\end{equation}
Then \(z_k\in \operatorname{span}(\one)^\perp\), and
\[
e_k=a_k\one+z_k = a_k\one+\proj e_k.
\]
Moreover, this decomposition is unique: if
\(e_k=\alpha\one+w\) with \(w\in \operatorname{span}(\one)^\perp\), then
\(\alpha=a_k\) and \(w=z_k\).
\end{lemma}
In the above lemma, \(z_k\) is the transverse component of the error and is
relevant to action orderings. The scalar \(a_k\) measures the magnitude of the
uniform all-ones component. Therefore, the decomposition in~\eqref{eq:empirical-error-decomposition} separates the
error into its policy-relevant projected part \(z_k\) and its policy-irrelevant
recentering part \(a_k\one\).

\begin{lemma}[Projected error recursion for standard Q-VI]
\label{lem:ordinary-projected-error-recursion}
The standard Q-VI recursion satisfies
\[
z_{k+1}=\bar A_{\mu_k}z_k,
\qquad
k\in\{0,1,2,\ldots\}.
\]
\end{lemma}
The next lemma gives the standard JSR bound for the projected switching system.
This bound identifies the transverse rate that the deflated Q-VI update will
later inherit as its associated switched-family JSR.
\begin{lemma}[Projected JSR bound~\cite{lee2026beyond}]
\label{lem:restricted-jsr-bound}
The projected JSR satisfies
\[
\bar\rho\le\gamma.
\]
\end{lemma}

This result follows from the projected-product argument in~\cite[Lemma~8]{lee2026beyond}.
The inequality can be strict when the state-action dynamics mix uniformly in
directions transverse to \(\operatorname{span}(\one)\). This projected
representation isolates the faster transverse dynamics from the slower all-ones
mode. Before introducing the deflated Q-VI update, we present the matching sharpness
of this slow all-ones mode for standard Q-VI.
\begin{proposition}[Sharpness of the classical rate for standard Q-VI]
\label{prop:ordinary-sharp}
For standard Q-VI, no uniform global convergence estimate to \(Q^\star\) can
have a rate smaller than \(\gamma\). In particular, if
\[
Q_0=Q^\star+c\one,
\qquad c\neq0,
\]
then standard Q-VI satisfies
\[
Q_k-Q^\star=\gamma^k c\one,
\qquad
\forall k\ge0.
\]
\end{proposition}

\section{Deflated Q-VI}\label{sec:deflated-vi}

Let us first define the Bellman residual and its empirical mean by
\[
r_k:=F(Q_k)-Q_k,
\qquad
\bar r_k:=\frac1{|\calS||\calA|}\one^\top r_k.
\]
The deflated Q-VI update is
\begin{equation}
\label{eq:deflated-vi}
Q_{k+1}
=
F(Q_k)
+
\frac{\gamma}{1-\gamma}\bar r_k\one,
\qquad
k\in\{0,1,2,\ldots\}.
\end{equation}
Equivalently,
\[
Q_{k+1}
=
F(Q_k)
+
\frac{\gamma}{1-\gamma}
\left(
\frac1{|\calS||\calA|}\one^\top(F(Q_k)-Q_k)
\right)\one .
\]
This update can be viewed as the all-ones version of the classical rank-one
residual correction of Bertsekas~\cite{bertsekas1995rankone}. In the
fixed-policy discounted linear case, that correction chooses a scalar residual
shift along a prescribed direction; when the direction is the all-ones vector,
the shift has exactly the same form as \cref{eq:deflated-vi}. Thus, the present
update is the Q-VI analogue of a classical rank-one recentering step.

The same formula is also closely related to R1-VI of Kolarijani et
al.~\cite{kolarijani2025rankone}. In their notation, \(d_k\) is the averaging
distribution, or residual-weighting vector, used at iteration \(k\), and
\(\langle d_k,x\rangle=d_k^\top x\). In the notation of the present paper, the
Q-function Bellman optimality operator is \(F\). The analogous Q-function
residual correction can therefore be written as
\[
Q_{k+1}=F(Q_k)+\frac{\gamma}{1-\gamma}
\langle d_k,F(Q_k)-Q_k\rangle\one.
\]
The uniform update in \cref{eq:deflated-vi} corresponds to fixing the averaging
vector to \((|\calS||\calA|)^{-1}\one\), while
\cref{sec:d-weighted-deflation} allows an arbitrary fixed state-action
distribution. The fixed-\(d\) analysis
should be distinguished from the adaptive \(d_k\) choices used in R1-VI/R1-QL~\cite{kolarijani2025rankone}.

\begin{algorithm}[H]
\caption{Deflated Q-VI}
\label{alg:deflated-vi}
\begin{algorithmic}[1]
\State Initialize \(Q_0\in\R^{|\calS||\calA|}\).
\For{\(k=0,1,2,\ldots\)}
    \State Compute the Bellman residual \(r_k=F(Q_k)-Q_k\).
    \State Compute its empirical mean \(\bar r_k=\frac1{|\calS||\calA|}\one^\top r_k\).
    \State Update \(Q_{k+1}=F(Q_k)+\frac{\gamma}{1-\gamma}\bar r_k\one\).
\EndFor
\end{algorithmic}
\end{algorithm}

The update separates the policy-relevant transverse convergence from the
policy-irrelevant all-ones direction. Standard Q-VI has a switching-system representation whose JSR equals \(\gamma\), because the all-ones direction is always an eigendirection
with eigenvalue \(\gamma\), and~\cref{prop:ordinary-sharp} shows that no error rate smaller than \(\gamma\) is possible for standard Q-VI. The
residual correction removes this all-ones direction in the associated switching system representation. Therefore, the convergence rate of the corresponding switching-system representation becomes the projected JSR \(\bar\rho\), the same quantity that controls the projected error dynamics in the previous section.
As shown below, this statement concerns value-error
recentering: the projected trajectory and the corresponding greedy policy sequence remain the same as in standard Q-VI initialized at the same point.

The first point to verify is that the correction term in deflated Q-VI does not introduce any spurious fixed point.
\begin{lemma}
\label{lem:deflated-fixed-point}
The deflated Q-VI map
\[
T_{\rm def}(Q)
:=
F(Q)
+
\frac{\gamma}{1-\gamma}
\left(
\frac1{|\calS||\calA|}\one^\top(F(Q)-Q)
\right)\one
\]
has the unique fixed point \(Q^\star\).
\end{lemma}
The above lemma shows that the deflated Q-VI map is consistent with the Bellman operator. The next lemma states the simple geometric effect:
compared with standard Q-VI, the deflated iterate differs only by an added term
\(c_k\one\).
\begin{lemma}[Scalar-shift equivalence to standard Q-VI]
\label{lem:scalar-shift-equivalence}
Let \(U_{k+1}=F(U_k)\) be standard Q-VI with \(U_0=Q_0\), and let
\(Q_k\) be generated by \cref{eq:deflated-vi} from the same initialization. Then
there exist scalars \(c_k\in\R\) such that
\[
Q_k=U_k+c_k\one,
\qquad \forall k\ge0.
\]
Consequently,
\[
\proj Q_k=\proj U_k,
\qquad
\pi_{Q_k}=\pi_{U_k},
\qquad \forall k\ge0,
\]
under the same tie-breaking rule.
\end{lemma}
\Cref{lem:scalar-shift-equivalence} shows that the residual correction does not
accelerate policy identification relative to standard Q-VI with the same
initialization. It only adds a uniform shift \(c_k\one\) to the
standard Q-VI iterate, so all state-action values move together by the same
amount. The benefit analyzed in this paper is faster convergence of the full
Q-function error after removing the all-ones mode.

\section{Exact switching representation of deflated Q-VI}
\label{sec:exact-deflated-dynamics}
The fixed-point property confirms that the modification is consistent with the
original Bellman equation. This section computes the exact switching system dynamics
induced by the deflated Q-VI map. This calculation is the key step that shows how the
all-ones direction is removed.

To proceed, let \(Q_k\) be generated by~\cref{eq:deflated-vi}, and define
\[
e_k:=Q_k-Q^\star,
\qquad
z_k:=\proj e_k .
\]
Since \(\proj\) is the orthogonal projection onto
\(\operatorname{span}(\one)^\perp\), the vector \(z_k\) is the component of the
error \(e_k\) orthogonal to the all-ones direction:
\[
z_k=\proj e_k\in \operatorname{span}(\one)^\perp .
\]
The remaining component of \(e_k\) lies in \(\operatorname{span}(\one)\). More
precisely, if
\[
a_k:=\frac1{|\calS||\calA|}\one^\top e_k,
\]
then the full error admits the orthogonal decomposition
\begin{equation}
\label{eq:deflated-error-decomposition-intro}
e_k
=
Q_k-Q^\star
=
a_k\one+\proj e_k
=
a_k\one+z_k .
\end{equation}
Here \(a_k\one\) is the all-ones component of the error, while
\(z_k\) is its transverse component in
\(\operatorname{span}(\one)^\perp\).
\begin{lemma}[Orthogonal decomposition of the Q-error]
\label{lem:error-decomposition}
Let \(Q_k\) be generated by~\cref{eq:deflated-vi}, and define
\[
e_k:=Q_k-Q^\star,
\qquad
z_k:=\proj e_k .
\]
Then \(z_k\in \operatorname{span}(\one)^\perp\). Moreover, if
\[
a_k:=\frac1{|\calS||\calA|}\one^\top e_k,
\]
then the full error admits the orthogonal decomposition
\begin{equation}
\label{eq:deflated-error-decomposition}
e_k
=
Q_k-Q^\star
=
a_k\one+\proj e_k
=
a_k\one+z_k .
\end{equation}
Here \(a_k\one\in\operatorname{span}(\one)\) is the all-ones component of the
error, while \(z_k\in\operatorname{span}(\one)^\perp\) is its orthogonal
transverse component. In addition, this decomposition is unique.
\end{lemma}

With this decomposition in place, we can now write the deflated Q-VI recursion in
separate coordinates. The next proposition shows that the projected component
\(z_k\) follows the same projected switching dynamics as standard Q-VI, while
the scalar component \(a_k\) is determined by \(z_k\) rather than by its own
previous value.
\begin{proposition}[Exact deflated switching dynamics]
\label{prop:exact-deflated-dynamics}
Let \(Q_k\) be generated by~\cref{eq:deflated-vi}, and define
\[
e_k:=Q_k-Q^\star,
\qquad
z_k:=\proj e_k,
\qquad
a_k:=\frac1{|\calS||\calA|}\one^\top e_k .
\]
Then
\[
z_{k+1}=\bar A_{\mu_k}z_k,
\qquad
k\in\{0,1,2,\ldots\},
\]
and
\[
a_{k+1}=\ell_{\mu_k}z_k,
\qquad
\ell_\mu:=\frac{1}{|\calS||\calA|(1-\gamma)}\one^\top A_\mu,
\qquad
k\in\{0,1,2,\ldots\}.
\]
Equivalently, they can be written by the following augmented system:
\[
\begin{bmatrix}
a_{k+1}\\ z_{k+1}
\end{bmatrix}
=
D_{\mu_k}
\begin{bmatrix}
a_k\\ z_k
\end{bmatrix},
\qquad
D_\mu:=
\begin{bmatrix}
0 & \ell_\mu\\
0 & \bar A_\mu
\end{bmatrix},
\qquad
k\in\{0,1,2,\ldots\}.
\]
\end{proposition}

The deflated Q-VI update keeps the projected component \(z_k\) evolving through the
projected switching system and recomputes the scalar all-ones component from
that projected error.
The block representation makes the role of deflation explicit. The transverse
state \(z_k\) evolves exactly as in projected standard Q-VI, whereas the scalar
component \(a_k\) is now driven only by \(z_k\) and no longer feeds itself
through an eigenvalue \(\gamma\).

\section{Provable JSR reduction}
\label{sec:jsr-reduction}

The augmented system formulation in~\cref{prop:exact-deflated-dynamics} immediately links the convergence rate of deflated
Q-VI to the projected switching system. Motivated by~\cref{prop:exact-deflated-dynamics}, let us define the associated augmented switching system family of deflated Q-VI by
\[
\mathcal D_{\rm def}
:=
\left\{
D_\mu=
\begin{bmatrix}
0 & \ell_\mu\\
0 & \bar A_\mu
\end{bmatrix}
:
\mu:\calS\to\Delta_{|\calA|}
\right\},
\]
where each \(D_\mu\) is regarded as a linear map on
\(\R\times\operatorname{span}(\one)^\perp\). Its JSR is denoted by
\[
\rho_{\rm def}:=\rho(\mathcal D_{\rm def}).
\]

\begin{theorem}[Associated full switched-family JSR of deflated Q-VI]
\label{thm:deflated-jsr}
The associated augmented switching family of the deflated Q-VI
iteration~\cref{eq:deflated-vi} satisfies
\[
\rho_{\rm def}=\bar\rho.
\]
Consequently, every actual deflated Q-VI trajectory has an exponential convergence
rate no larger than \(\bar\rho\), in the switched-family sense made explicit in \cref{thm:explicit-deflated-rate}. Since \(\bar\rho\le\gamma\), this certified
rate is no worse than the classical worst-case full-error rate of standard Q-VI identified in~\cref{prop:qvi-full-jsr,prop:ordinary-sharp}. If \(\bar\rho<\gamma\), the
certified switching system rate is strictly smaller than \(\gamma\).
\end{theorem}
The equality \(\rho_{\rm def}=\bar\rho\) is most informative when the projected
rate is strictly smaller than the discount factor. The following elementary MDP
shows that such a strict separation can occur.
\begin{example}
\label{ex:strict-projected-rate}
Consider a discounted MDP with one state and two actions. The next state is the
same state with probability one under both actions, and the rewards can be
chosen arbitrarily. Then the state-action vector has dimension two. For the two
deterministic policies, let \(e_1,e_2\) denote the standard basis vectors and
\(\one=(1,1)^\top\). The corresponding linear parts are
\[
A_1=\gamma\one e_1^\top,
\qquad
A_2=\gamma\one e_2^\top .
\]
If \(z\in\operatorname{span}(\one)^\perp\), then \(z_1+z_2=0\), and
\[
A_i z=\gamma z_i\one,
\qquad i=1,2.
\]
Thus \(A_i z\) lies entirely in \(\operatorname{span}(\one)\), so the transverse
projection removes it:
\[
\proj A_i\proj z=0,
\qquad i=1,2.
\]
Hence, \(\bar A_1=\bar A_2=0\), and therefore
\[
\bar\rho=0<\gamma
\]
for every \(\gamma\in(0,1)\). In this example, standard Q-VI still has JSR \(\gamma\) by~\cref{prop:qvi-full-jsr}, whereas the
JSR associated with the deflated Q-VI is zero by~\cref{thm:deflated-jsr}.
\end{example}

Together with \cref{prop:ordinary-sharp}, this example separates the convergence of standard Q-VI error from the convergence of deflated Q-VI:
standard Q-VI retains the sharp \(\gamma\) all-ones mode, while deflated Q-VI
removes it from the associated switching dynamics.
The following result presents an explicit finite-time error bound for the deflated Q-VI recursion.
\begin{theorem}\label{thm:explicit-deflated-rate}
Let us consider~\cref{eq:deflated-vi}. Fix any
\(\varepsilon\in(0,1-\bar\rho)\), and define
\[
\beta_\varepsilon:=\bar\rho+\varepsilon .
\]
Then there exists \(C_{\beta_\varepsilon}>0\) such that, for all \(k\ge1\),
\[
\|Q_k-Q^\star\|_2
\le
C_{\beta_\varepsilon}\beta_\varepsilon^{k-1}\|\proj(Q_0-Q^\star)\|_2.
\]
Consequently,
\[
\limsup_{k\to\infty}\|Q_k-Q^\star\|_2^{1/k}\le\bar\rho.
\]
If \(\bar\rho<\gamma\), then deflated Q-VI has a certified full-error convergence rate strictly smaller than the classical worst-case full-error rate of standard Q-VI.
\end{theorem}
The next lemma gives a simple explicit estimate for this constant, which can be used
when a completely numerical convergence bound is desired.
\begin{lemma}[Basic bound for the scalar lifting constant]
\label{lem:scalar-lifting-bound}
For
\[
\ell_\mu:=\frac{1}{|\calS||\calA|(1-\gamma)}\one^\top A_\mu,
\]
one has
\[
L_\ell:=\sup_\mu\|\ell_\mu\|_2\le \frac{\gamma}{1-\gamma}.
\]
Consequently, if the projected switching family satisfies
\[
\|\bar A_{\mu_{m-1}}\cdots\bar A_{\mu_0}\|_2
\le c_{\beta_\varepsilon}\beta_\varepsilon^m,
\qquad
\beta_\varepsilon:=\bar\rho+\varepsilon,
\]
for all \(m\ge0\) and all stochastic-policy switching sequences, then the constant
in the full-error estimate can be chosen as
\[
C_{\beta_\varepsilon}
:=c_{\beta_\varepsilon}
\left(\frac{\sqrt{|\calS||\calA|}\,\gamma}{1-\gamma}+\beta_\varepsilon\right).
\]
Moreover, if \(N_{\beta_\varepsilon}\) is the finite threshold in
\cref{lem:jsr-product-bound} for this projected family, then one may use the
coarse bound
\[
 c_{\beta_\varepsilon}
 \le
 \max\left\{1,\max_{0\le m<N_{\beta_\varepsilon}}
 \left(\frac{\gamma\sqrt{|\calS||\calA|}}{\beta_\varepsilon}\right)^m\right\},
\]
and hence
\[
C_{\beta_\varepsilon}
\le
\max\left\{1,\max_{0\le m<N_{\beta_\varepsilon}}
\left(\frac{\gamma\sqrt{|\calS||\calA|}}{\beta_\varepsilon}\right)^m\right\}
\left(\frac{\sqrt{|\calS||\calA|}\,\gamma}{1-\gamma}+\beta_\varepsilon\right).
\]
\end{lemma}

Pure uniform-shift errors are eliminated immediately, while general initial conditions converge at
the rate of the projected switching family JSR.
\begin{corollary}[Pure all-ones errors vanish in one step]
\label{cor:pure-ones}
If \(Q_0-Q^\star\in\operatorname{span}(\one)\), then deflated Q-VI satisfies
\[
Q_1=Q^\star.
\]
\end{corollary}

\begin{corollary}[Iteration complexity for Q-function error]
\label{cor:iteration-complexity}
Let \(\varepsilon_Q>0\), fix \(\varepsilon\in(0,1-\bar\rho)\), and define
\(\beta_\varepsilon:=\bar\rho+\varepsilon\). If
\(\proj(Q_0-Q^\star)=0\), then \(Q_1=Q^\star\). Otherwise, the bound
\[
k
\ge
1+
\left\lceil
\frac{
\ln\left(C_{\beta_\varepsilon}\|\proj(Q_0-Q^\star)\|_2/\varepsilon_Q\right)
}{
-\ln \beta_\varepsilon
}
\right\rceil
\]
implies
\[
\|Q_k-Q^\star\|_2\le\varepsilon_Q.
\]
\end{corollary}

\section{Generalized deflated Q-VI}\label{sec:d-weighted-deflation}

The uniform residual average in~\cref{eq:deflated-vi} is convenient but not
essential.
Let
\[
d=(d_i)_{i\in\calS\times\calA}\in\Delta_{|\calS||\calA|},
\qquad
 d^\top\one=1,
\]
be a fixed state-action distribution, possibly with zero components.
In this section, we consider the more general deflated Q-VI
\begin{equation}
\label{eq:d-weighted-deflated-vi}
Q_{k+1} = F(Q_k)+\frac{\gamma}{1-\gamma} d^\top(F(Q_k)-Q_k)\one=T_d(Q_k), \qquad k\in\{0,1,2,\ldots\},
\end{equation}
where $T_d$ is the generalized deflated Q-VI map
\begin{equation}\label{eq:d-weighted-deflated-map}
T_d(Q):=F(Q) + \frac{\gamma}{1-\gamma} d^\top(F(Q)-Q)\one.
\end{equation}
Nonuniform choices of \(d\) are natural in randomized implementations
and reinforcement-learning settings, where state-action pairs may be sampled
according to a nonuniform visitation or exploration distribution.
When $d = \frac1{|\calS||\calA|}\one$, the update in~\cref{eq:d-weighted-deflated-vi} reduces to the deflated Q-VI in the previous section.
In the above mapping, the correction direction remains \(\one\). Replacing the correction direction \(\one\) by \(d\) generally does not remove the all-ones mode and can
change the JSR.

To proceed, let us define
\[
\boldsymbol{\Pi}_d:=I-\one d^\top
\]
and
\[
W_d:=\operatorname{span}(d)^\perp
=\{x\in\R^{|\calS||\calA|}:d^\top x=0\},
\]
which is the subspace of state-action vectors whose \(d\)-weighted average is zero. Since
\(d^\top\one=1\), we have
\[
\boldsymbol{\Pi}_d^2
=(I-\one d^\top)^2
=I-2\one d^\top+\one(d^\top\one)d^\top
=I-\one d^\top
=\boldsymbol{\Pi}_d.
\]
Thus applying \(\boldsymbol{\Pi}_d\) twice is the same as applying it once, so
\(\boldsymbol{\Pi}_d\) is a projection. Moreover,
\[
\operatorname{range}(\boldsymbol{\Pi}_d)=W_d,
\qquad
\ker(\boldsymbol{\Pi}_d)=\operatorname{span}(\one),
\]
because
\[
d^\top\boldsymbol{\Pi}_d x=d^\top x-d^\top\one\,d^\top x=0
\]
for every \(x\), and
\(\boldsymbol{\Pi}_d x=0\) holds exactly when \(x=(d^\top x)\one\). Therefore,
\[
\R^{|\calS||\calA|}=W_d\oplus\operatorname{span}(\one),
\]
and \(\boldsymbol{\Pi}_d\) maps a vector to its \(W_d\)-component while removing
its \(\operatorname{span}(\one)\)-component. This projection is generally
\emph{oblique}: the two subspaces in the direct sum need not be orthogonal. It
is orthogonal only in the special case where the hyperplane \(W_d\) is
perpendicular to \(\operatorname{span}(\one)\), which occurs when \(d\) is
proportional to \(\one\). This is the standard linear-algebra meaning of an
oblique projection~\cite{meyer2000matrix}.
In particular,
\[
\boldsymbol{\Pi}_d\one=0,
\qquad
d^\top\boldsymbol{\Pi}_d=0,
\]
and every error vector admits the unique decomposition
\[
e=a\one+z,
\qquad
 a=d^\top e,
\qquad
 z=\boldsymbol{\Pi}_d e\in W_d.
\]
Throughout this section, the symbols \(a\), \(z\), \(a_k\), and \(z_k\) refer
to this \(d\)-adapted decomposition; they are not the same coordinates as the
empirical-mean decomposition used in the preceding sections.
The identity \(z=\boldsymbol{\Pi}_d e\in W_d\) means that \(z\) is the part of the
error remaining after removing the uniform component \(a\one\). Indeed,
\[
z=\boldsymbol{\Pi}_d e=e-\one d^\top e=e-a\one,
\]
and therefore
\[
d^\top z=d^\top e-d^\top\one\, d^\top e=a-a=0.
\]
Thus \(z\) has zero \(d\)-average, while \(a\one\) is the uniform-shift
component measured by the same averaging functional \(d^\top\). The
decomposition is unique because applying \(d^\top\) to
\(e=a\one+z\) gives \(d^\top e=a\), since \(d^\top\one=1\) and
\(d^\top z=0\).

The next lemma shows that this
change of complement does not change the fixed point.
\begin{lemma}[Fixed point of the \(d\)-weighted deflated map]
\label{lem:d-weighted-fixed-point}
The map \(T_d\) in \cref{eq:d-weighted-deflated-map} has the unique fixed point
\(Q^\star\).
\end{lemma}
With uniqueness of the fixed point established, the corresponding switching system model can be derived in the oblique decomposition induced by \(d\).
\begin{proposition}
\label{prop:d-weighted-dynamics}
Let \(Q_k\) be generated by \cref{eq:d-weighted-deflated-vi}, and let
\(e_k=Q_k-Q^\star\). For the stochastic-policy representation in
\cref{lem:exact-representation}, define the \(d\)-adapted coordinates
\[
z_k:=\boldsymbol{\Pi}_d e_k,
\qquad
 a_k:=d^\top e_k.
\]
Then
\begin{align*}
 z_{k+1}
 &=
 \boldsymbol{\Pi}_d A_{\mu_k}\boldsymbol{\Pi}_d\, z_k,
 \qquad
 k\in\{0,1,2,\ldots\},
 \\
 a_{k+1}
 &=
 \ell_{\mu_k}z_k,
 \qquad
 \ell_\mu:=\frac{1}{1-\gamma}d^\top A_\mu,
 \qquad
 k\in\{0,1,2,\ldots\}.
\end{align*}
Equivalently,
\[
\begin{bmatrix}
 a_{k+1}\\ z_{k+1}
\end{bmatrix}
=
D_{\mu_k}
\begin{bmatrix}
 a_k\\ z_k
\end{bmatrix},
\qquad
D_{\mu}:=
\begin{bmatrix}
0 & \ell_\mu\\
0 & \boldsymbol{\Pi}_d A_\mu\boldsymbol{\Pi}_d
\end{bmatrix},
\qquad
k\in\{0,1,2,\ldots\},
\]
where the lower-right block is regarded as a linear map from \(W_d\) to
\(W_d\). Thus the all-ones component again has zero autonomous diagonal
dynamics.
\end{proposition}
Here \(a_k\), \(z_k\), \(\ell_\mu\), and \(D_\mu\) are the \(d\)-adapted
quantities used only in this section. They should be distinguished from the
empirical-mean coordinates and the uniform-deflation blocks introduced earlier.
The full error is
\[
e_k=Q_k-Q^\star=a_k\one+z_k,
\qquad
a_k=d^\top e_k,
\qquad
z_k=\boldsymbol{\Pi}_d e_k\in W_d .
\]

The projected block is now written on the zero-\(d\)-average subspace \(W_d\)
instead of the zero-sum subspace \(\operatorname{span}(\one)^\perp\). The following theorem shows that
these two projected families are similar, and therefore have the same JSR.

\begin{theorem}[JSR of distribution-weighted residual deflation]
\label{thm:d-weighted-jsr}
For any fixed \(d\in\Delta_{|\calS||\calA|}\), the associated full stochastic-policy
switched-family JSR of \cref{eq:d-weighted-deflated-vi} is
\[
\rho_d=\bar\rho\le\gamma.
\]
Consequently, the distribution-weighted deflated Q-VI map has the same
associated switched-family rate as the uniform deflated Q-VI map in
\cref{eq:deflated-vi}.
\end{theorem}
\section{Numerical illustrations}
\label{sec:numerical-illustrations}

This section illustrates the effect of deflated Q-VI recentering on finite
discounted MDPs. Three examples are used: a small synthetic MDP, a stochastic
\texttt{FrozenLake} benchmark, and the standard \texttt{Taxi-v3} benchmark. The
examples isolate the slow common-shift mode of standard Q-VI and show that the
residual correction removes this mode while preserving the optimal fixed point.
Because \cref{lem:scalar-shift-equivalence} shows that the greedy-policy
sequence is unchanged relative to standard Q-VI, the plots focus on full
Q-function error rather than policy-identification speed.

First, consider a small MDP with four states,
\(\mathcal S = \{1,2,3,4\}\), two actions, \(\mathcal A = \{1,2\}\), and
discount factor \(\gamma = 0.95\). The reward function is defined by
\[
R =
\begin{bmatrix}
0.7666 & 1.3368 \\
0.0000 & 0.0000 \\
0.0000 & 0.0000 \\
1.0406 & 0.0000
\end{bmatrix},
\]
where the entry in row \(s\) and column \(a\) is \(R(s,a)\). The transition
kernel is represented by two transition matrices \(P_1\) and \(P_2\), one for
each action:
\[
P_1 =
\begin{bmatrix}
1.0000 & 0.0000 & 0.0000 & 0.0000 \\
0.0000 & 0.0000 & 0.0000 & 1.0000 \\
0.6074 & 0.0000 & 0.3926 & 0.0000 \\
0.0000 & 0.0000 & 0.6765 & 0.3235
\end{bmatrix},
\qquad
P_2 =
\begin{bmatrix}
0.5483 & 0.4517 & 0.0000 & 0.0000 \\
1.0000 & 0.0000 & 0.0000 & 0.0000 \\
1.0000 & 0.0000 & 0.0000 & 0.0000 \\
0.0000 & 0.0000 & 1.0000 & 0.0000
\end{bmatrix},
\]
where the \(s\)-th row of \(P_a\) gives the distribution over next states when
action \(a\) is chosen at state \(s\).

Three methods are compared: standard Q-VI, deflated Q-VI with the uniform
distribution, and deflated Q-VI with a nonuniform distribution. The optimal
Q-function, computed by policy iteration, is
\[
Q^\star =
\begin{bmatrix}
18.5393 & 18.7081 \\
17.0925 & 17.7727 \\
17.4238 & 17.7727 \\
17.9921 & 16.8840
\end{bmatrix}.
\]

Figure~\ref{fig:example-full-error} plots the full Q-function error
\(\|Q_k-Q^\star\|_\infty\) over \(80\) iterations. Standard Q-VI has empirical
tail rate \(0.95\), matching the discount factor \(\gamma\), because the
dominant remaining error is the common-shift component. Both deflated variants
eliminate this slow mode and reduce the error by many orders of magnitude
within a small number of iterations.

\begin{figure}[t]
\centering
\includegraphics[width=0.72\linewidth]{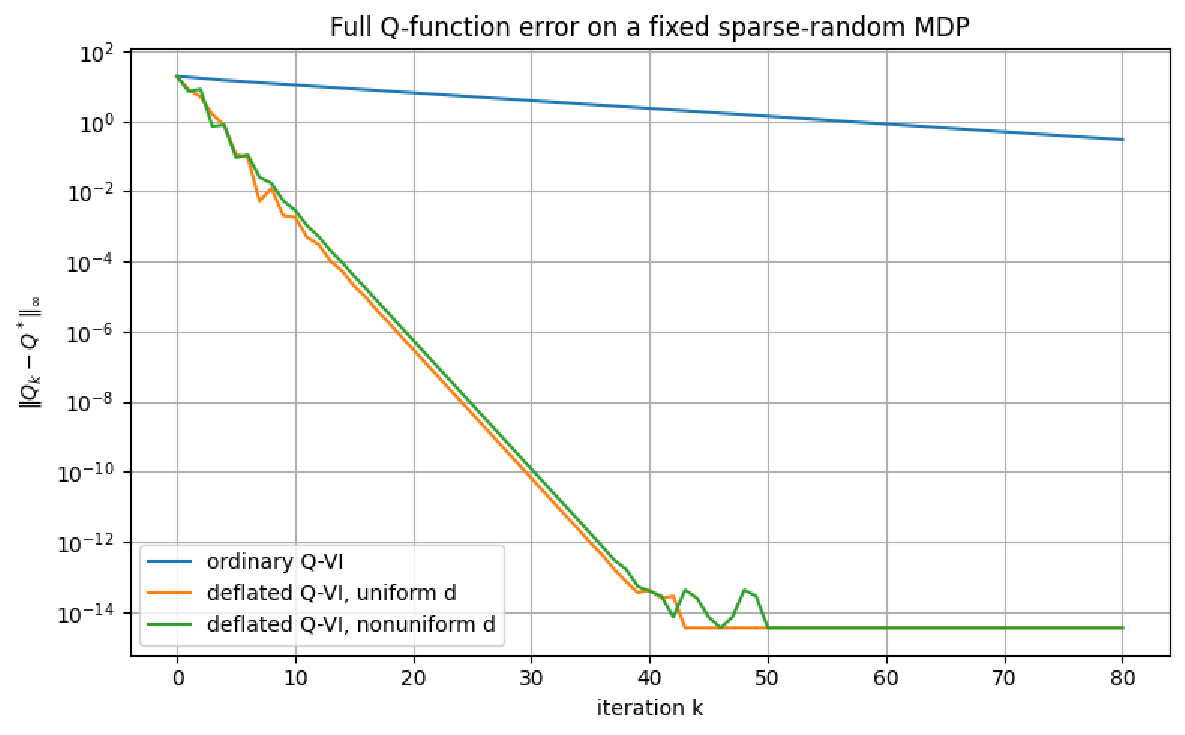}
\caption{
Full Q-function error for standard Q-VI and deflated Q-VI with uniform and
nonuniform distributions. The deflated Q-VI updates remove the slow all-ones
component and converge substantially faster than standard Q-VI.
}
\label{fig:example-full-error}
\end{figure}

Although the greedy policy induced by the initialization is already optimal in
this example, the value error remains large for standard Q-VI because of the common
shift. This demonstrates that policy accuracy and value accuracy can behave
differently, and that deflated Q-VI specifically accelerates
convergence of the Q-function itself.

Next, standard Q-VI is compared with deflated Q-VI on a standard
stochastic \texttt{FrozenLake} benchmark. In the experiment, \(\gamma=0.99\),
the distribution \(d\) is uniform over state-action pairs, and the initial
condition is \(Q_0=10\mathbf 1\).

Figure~\ref{fig:frozenlake-full-q-error} shows the evolution of the Q-function
error \(\|Q_k-Q^\star\|_\infty\) on a logarithmic scale. The horizontal axis is
the iteration index \(k\), and the vertical axis is the sup-norm error with
respect to the optimal Q-function \(Q^\star\).

The deflated Q-VI update decreases the full error much faster than standard
Q-VI. In particular, the standard iteration exhibits a slow decay phase, while
the deflated Q-VI iteration maintains a significantly steeper convergence
trend. This behavior is consistent with the fact that the residual correction
removes the slowly decaying component aligned with the all-ones direction. As a
result, the full Q-function itself approaches \(Q^\star\) substantially faster.

Thus, the \texttt{FrozenLake} experiment illustrates the main advantage of the
deflated Q-VI update: although both methods are based on the same Bellman
optimality operator, the deflation term accelerates convergence in the full
sup-norm sense by eliminating the residual constant-shift component.

\begin{figure}[t]
\centering
\includegraphics[width=0.72\linewidth]{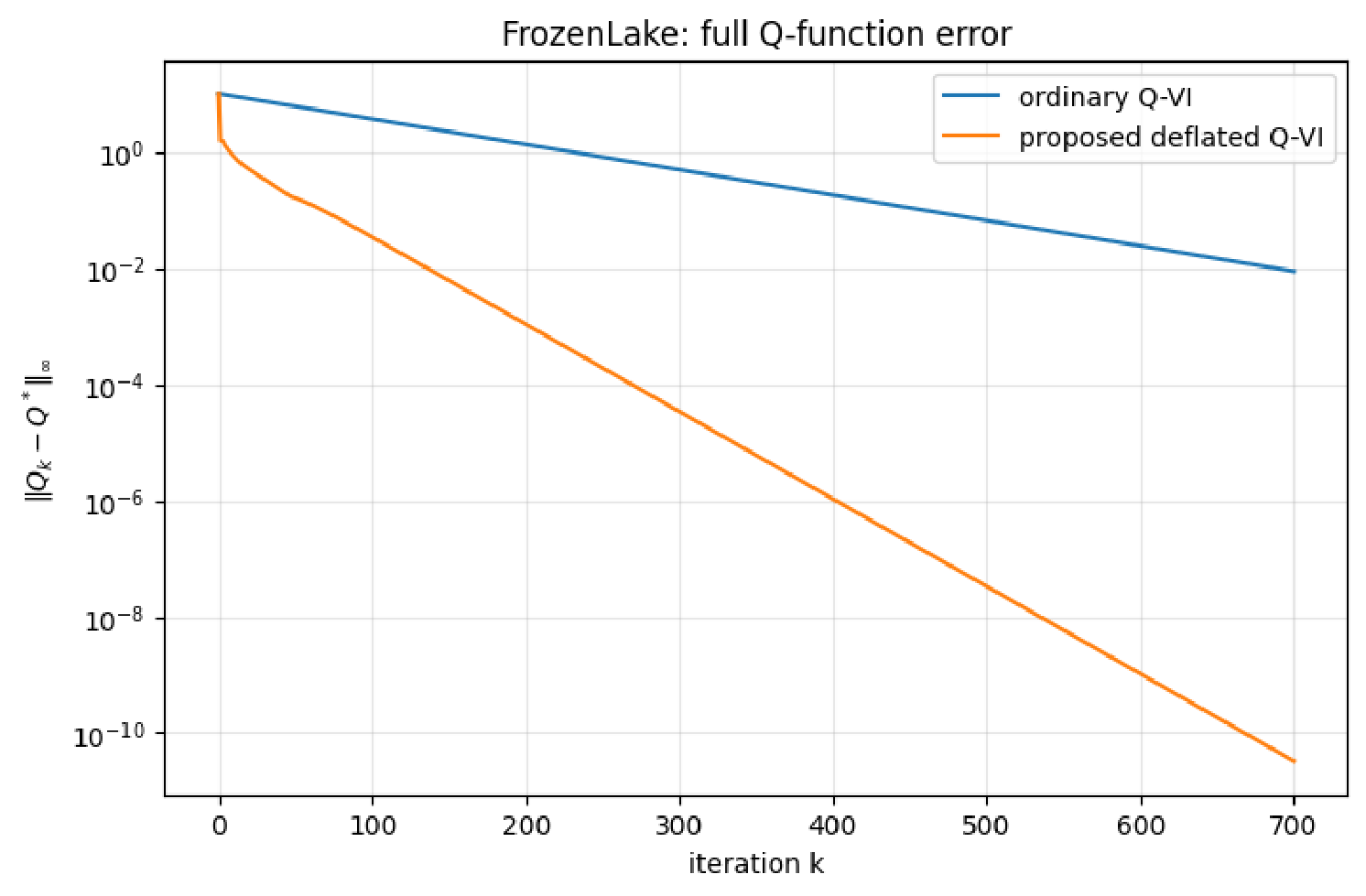}
\caption{
Convergence of the full Q-function error
\(\|Q_k-Q^\star\|_\infty\) for standard Q-VI and deflated Q-VI on the
\texttt{FrozenLake} benchmark. The deflated Q-VI update achieves a
substantially faster decay of the full sup-norm error.
}
\label{fig:frozenlake-full-q-error}
\end{figure}

Finally, the same deflation idea is evaluated on the standard \texttt{Taxi-v3}
benchmark. This example is larger and more structured than the preceding small
MDP and \texttt{FrozenLake} examples, and therefore provides a useful test of
whether the acceleration effect persists in a more practical finite-state
control problem. Standard Q-VI is compared with deflated Q-VI, and the
full Q-function error \(\|Q_k-Q^\star\|_\infty\) is measured with respect to a
numerically computed optimal Q-function \(Q^\star\).

Figure~\ref{fig:taxi-full-q-error} plots the error over \(700\) iterations on a
logarithmic scale. The standard Q-VI curve decreases steadily but slowly,
remaining visibly above the numerical precision level even after hundreds of
iterations. This slow tail behavior is again caused by the constant-shift
component, which is contracted only at the discount-factor rate.

In contrast, deflated Q-VI removes this slow mode. After a short
initial transient, the error drops rapidly to approximately machine precision,
around \(10^{-13}\), and remains at that level for the rest of the iterations.
Thus, on \texttt{Taxi-v3}, the separation between the two methods is especially
clear: the standard iteration still has a non-negligible full Q-function error
at the end of the plotted horizon, whereas the deflated Q-VI iteration has already
reached the numerical accuracy floor.

This experiment further supports the central claim of the deflated Q-VI
update. The deflation does not change the Bellman optimality fixed point, but it
removes the slowly decaying all-ones component from the iteration. Consequently,
the full Q-function converges to \(Q^\star\) substantially faster in the
sup-norm sense.

\begin{figure}[t]
\centering
\includegraphics[width=0.72\linewidth]{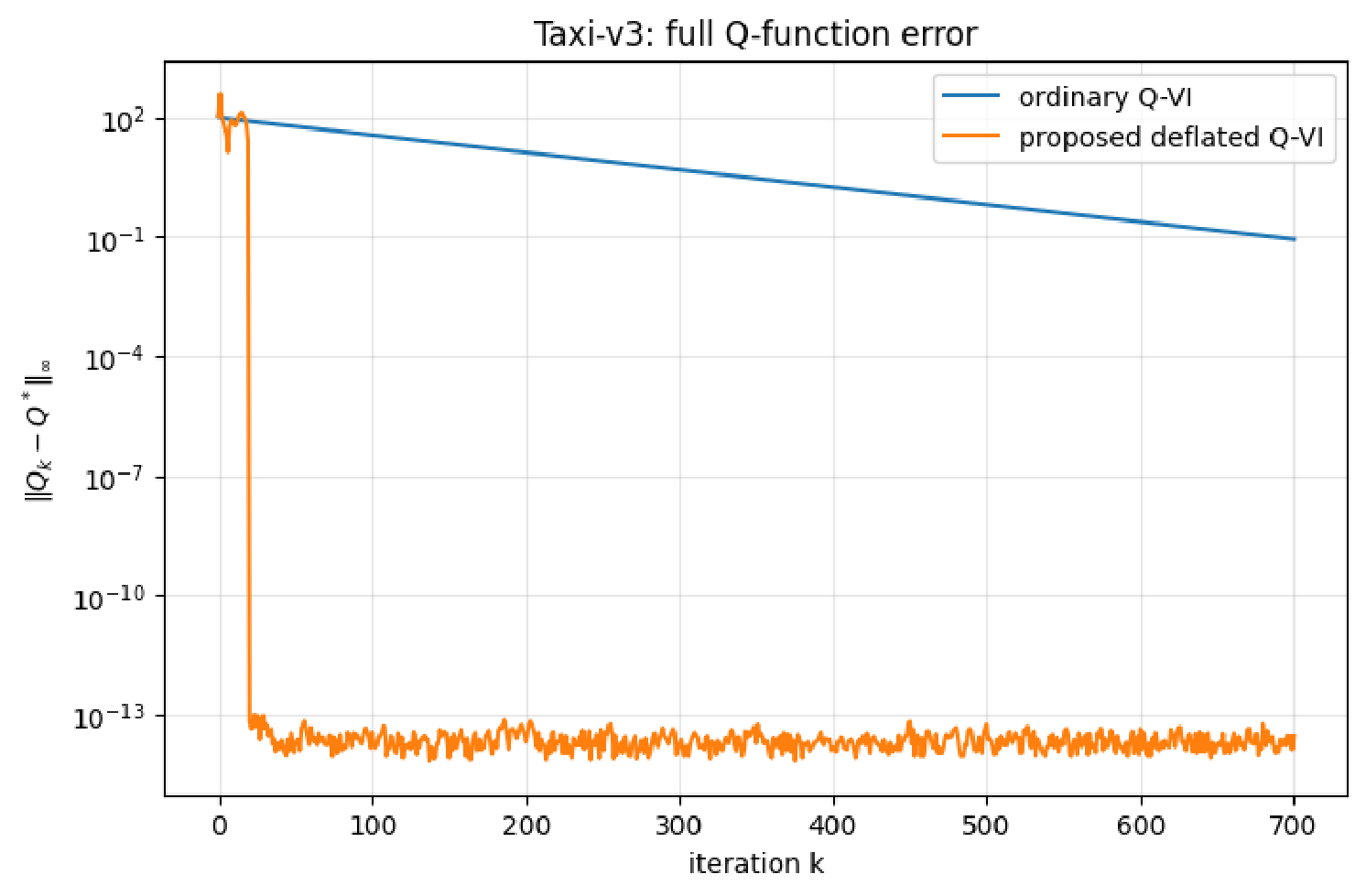}
\caption{
Full Q-function error
\(\|Q_k-Q^\star\|_\infty\) for standard Q-VI and deflated Q-VI on the
\texttt{Taxi-v3} benchmark. Standard Q-VI exhibits a slow tail decay, while the
deflated Q-VI update removes the common-shift mode and reaches numerical
precision after a short transient.
}
\label{fig:taxi-full-q-error}
\end{figure}

\section{Conclusion}

This paper developed a joint spectral radius framework for rank-one deflated
Q-VI in discounted MDP control. The analysis shows that the standard Q-VI
switching family has JSR exactly equal to the discount factor \(\gamma\), because
the all-ones vector is a common invariant direction of every subsystem. After
this direction is removed, the relevant transverse dynamics are governed by the
projected switching family with projected JSR \(\bar\rho\), which can be strictly
smaller than \(\gamma\).

The rank-one residual correction removes the autonomous evolution of the
all-ones component from the associated deflated switching system, so the
resulting switched-family rate is \(\bar\rho\). The same projected rate is
preserved when the residual average is taken with respect to any fixed
state-action distribution, through an oblique-projection similarity argument.
At the same time, the deflated iterate differs from the standard Q-VI iterate,
initialized at the same point, only by a scalar multiple of \(\one\). Therefore,
the projected trajectory and the greedy-policy sequence are unchanged. The main
benefit of deflation is thus not a different decision-making trajectory, but a
sharper JSR-based description of the convergence geometry and a value-error
recentering mechanism for full Q-function convergence.

\appendix

\section{Proof of convexification invariance of the JSR}
\label{app:proof-jsr-convexification}

\noindent\textbf{Restatement of \cref{lem:jsr-convexification}.}
For every finite matrix family \(\calH\),
\[
\rho(\co(\calH))=\rho(\calH).
\]

\begin{proof}
Fix the same arbitrary submultiplicative matrix norm \(\|\cdot\|\) as in
\cref{def:jsr}. Since \(\calH\subseteq\co(\calH)\),
\(\rho(\calH)\le\rho(\co(\calH))\). It remains to prove the reverse inequality.

Fix a product length \(k\ge1\). Since \(\calH\) is finite, write
\(\calH=\{H_1,\ldots,H_M\}\). For arbitrary
\(B_1,\ldots,B_k\in\co(\calH)\), there exist coefficients
\(\lambda_{j,i}\ge0\), \(i=1,\ldots,M\), satisfying
\(\sum_{i=1}^M\lambda_{j,i}=1\), such that
\[
B_j=\sum_{i=1}^M\lambda_{j,i}H_i,
\qquad j=1,\ldots,k.
\]
Expanding the product gives
\begin{align*}
B_k\cdots B_1
&=
\left(\sum_{i_k=1}^M\lambda_{k,i_k}H_{i_k}\right)
\cdots
\left(\sum_{i_1=1}^M\lambda_{1,i_1}H_{i_1}\right) \\
&=
\sum_{i_1=1}^M\cdots\sum_{i_k=1}^M
\left(\prod_{j=1}^k\lambda_{j,i_j}\right)
H_{i_k}\cdots H_{i_1}.
\end{align*}
The coefficients in this expansion form a convex combination because
\[
\sum_{i_1=1}^M\cdots\sum_{i_k=1}^M
\prod_{j=1}^k\lambda_{j,i_j}
=
\prod_{j=1}^k\sum_{i_j=1}^M\lambda_{j,i_j}
=1 .
\]
Define the finite-length product maximum
\[
M_k(\calH):=
\max_{A_1,\ldots,A_k\in\calH}\|A_k\cdots A_1\|.
\]
The maximum exists because \(\calH\) is finite. By convexity of the norm,
\begin{align*}
\|B_k\cdots B_1\|
&\le
\sum_{i_1=1}^M\cdots\sum_{i_k=1}^M
\left(\prod_{j=1}^k\lambda_{j,i_j}\right)
\|H_{i_k}\cdots H_{i_1}\| \\
&\le
\sum_{i_1=1}^M\cdots\sum_{i_k=1}^M
\left(\prod_{j=1}^k\lambda_{j,i_j}\right) M_k(\calH) \\
&=M_k(\calH).
\end{align*}
Thus, for every \(B_1,\ldots,B_k\in\co(\calH)\),
\[
\|B_k\cdots B_1\|
\le
\max_{A_1,\ldots,A_k\in\calH}\|A_k\cdots A_1\|.
\]
Taking the supremum over \(B_1,\ldots,B_k\in\co(\calH)\) gives
\[
\sup_{B_1,\ldots,B_k\in\co(\calH)}\|B_k\cdots B_1\|
\le
\max_{A_1,\ldots,A_k\in\calH}\|A_k\cdots A_1\|.
\]
Taking the \(k\)-th root and letting \(k\to\infty\) in the definition of the JSR
yields \(\rho(\co(\calH))\le\rho(\calH)\). Therefore the two JSRs are equal.
\end{proof}

\section{Proof of the uniform exponential product bound}
\label{app:proof-jsr-product-bound}

\noindent\textbf{Restatement of \cref{lem:jsr-product-bound}.}
Let \(\|\cdot\|\) be any fixed submultiplicative matrix norm. If
\(\rho({\calH})<1\), then every switched product is uniformly exponentially
stable in that norm: for every \(\varepsilon\in(0,1-\rho(\calH))\), with
\[
\beta_\varepsilon:=\rho(\calH)+\varepsilon,
\]
there exists \(C_{\beta_\varepsilon}>0\), depending on
\(\beta_\varepsilon\) and on the chosen norm, such that for every \(k\ge0\) and
every switching word \(\sigma_0,\ldots,\sigma_{k-1}\) with
\[
A_j:=A_{\sigma_j}\in\calH,
\qquad j=0,\ldots,k-1,
\]
one has
\[
\left\|A_{k-1}\cdots A_0\right\|
\le
C_{\beta_\varepsilon}\beta_\varepsilon^k,
\]
where the product is interpreted as the identity when \(k=0\).

\begin{proof}
Fix the chosen submultiplicative matrix norm and define
\[
M_k:=
\sup_{A_0,\ldots,A_{k-1}\in\calH}
\left\|A_{k-1}\cdots A_0\right\|,
\qquad
M_0:=1 .
\]
By the definition of the JSR,
\[
\lim_{k\to\infty} M_k^{1/k}=\rho(\calH).
\]
Fix \(\varepsilon\in(0,1-\rho(\calH))\) and define
\(\beta_\varepsilon:=\rho(\calH)+\varepsilon\). Since the limit is strictly
smaller than \(\beta_\varepsilon\), there exists an integer
\(N_{\beta_\varepsilon}\) such that
\[
M_k^{1/k}\le \beta_\varepsilon,
\qquad
\forall k\ge N_{\beta_\varepsilon}.
\]
Equivalently, \(M_k\le \beta_\varepsilon^k\) for all
\(k\ge N_{\beta_\varepsilon}\). For the finitely many shorter lengths, define
\[
C_{\beta_\varepsilon}:=
\max\left\{1,\max_{0\le k<N_{\beta_\varepsilon}}\frac{M_k}{\beta_\varepsilon^k}\right\}.
\]
Then \(C_{\beta_\varepsilon}<\infty\). If \(k<N_{\beta_\varepsilon}\), the
definition of \(C_{\beta_\varepsilon}\) gives
\(M_k\le C_{\beta_\varepsilon}\beta_\varepsilon^k\). If
\(k\ge N_{\beta_\varepsilon}\), then
\(M_k\le \beta_\varepsilon^k\le C_{\beta_\varepsilon}\beta_\varepsilon^k\).
Therefore, for every length \(k\ge0\) and every switching word
\(\sigma_0,\ldots,\sigma_{k-1}\), with selected matrices
\(A_j:=A_{\sigma_j}\in\calH\),
\[
\left\|A_{k-1}\cdots A_0\right\|
\le M_k
\le C_{\beta_\varepsilon}\beta_\varepsilon^k .
\]
This proves the claim.
\end{proof}

\section{Proof of the block upper-triangular JSR formula}
\label{app:proof-block-upper-triangular-jsr}

\noindent\textbf{Restatement of \cref{lem:block-upper-triangular-jsr}.}
Let
\[
\mathcal M:=
\left\{
\begin{bmatrix}
B_i & C_i\\
0 & D_i
\end{bmatrix}:i\in\mathcal I
\right\}
\]
be a bounded family of block upper-triangular matrices. Let
\(\mathcal B:=\{B_i:i\in\mathcal I\}\) and
\(\mathcal D:=\{D_i:i\in\mathcal I\}\). Then
\[
\rho(\mathcal M)=\max\{\rho(\mathcal B),\rho(\mathcal D)\}.
\]

\begin{proof}
This is a standard property of block triangular matrix families. The proof is
included to make clear why the off-diagonal blocks do not affect the exponential
rate. Use a block-compatible submultiplicative norm, for example the operator
norm induced by \(\|(x,y)\|:=\max\{\|x\|,\|y\|\}\). Since the JSR is independent
of the chosen submultiplicative norm, this choice is without loss of generality.

For the lower bound, fix a switching word \(i_0,\ldots,i_{k-1}\) and write
\[
M_{i_j}:=
\begin{bmatrix}
B_{i_j} & C_{i_j}\\
0 & D_{i_j}
\end{bmatrix}.
\]
The product remains block upper triangular:
\[
M_{i_{k-1}}\cdots M_{i_0}
=
\begin{bmatrix}
B_{i_{k-1}}\cdots B_{i_0} & *\\
0 & D_{i_{k-1}}\cdots D_{i_0}
\end{bmatrix}.
\]
With the chosen block norm,
\[
\left\|M_{i_{k-1}}\cdots M_{i_0}\right\|
\ge
\left\|B_{i_{k-1}}\cdots B_{i_0}\right\|,
\qquad
\left\|M_{i_{k-1}}\cdots M_{i_0}\right\|
\ge
\left\|D_{i_{k-1}}\cdots D_{i_0}\right\|.
\]
Taking the supremum over switching words, then the \(k\)-th root and the limit,
gives
\[
\rho(\mathcal M)\ge \rho(\mathcal B),
\qquad
\rho(\mathcal M)\ge \rho(\mathcal D).
\]
Thus
\[
\rho(\mathcal M)\ge \max\{\rho(\mathcal B),\rho(\mathcal D)\}.
\]

For the upper bound, fix any
\[
q>\max\{\rho(\mathcal B),\rho(\mathcal D)\}.
\]
Applying \cref{lem:jsr-product-bound} to the scaled families
\(q^{-1}\mathcal B\) and \(q^{-1}\mathcal D\) gives constants
\(K_B,K_D\ge1\) such that every length-\(m\) product satisfies
\[
\left\|B_{i_{m-1}}\cdots B_{i_0}\right\|\le K_B q^m,
\qquad
\left\|D_{i_{m-1}}\cdots D_{i_0}\right\|\le K_D q^m,
\qquad
m\ge0,
\]
where the length-zero product is the identity. Since \(\mathcal M\) is bounded,
\[
K_C:=\sup_{i\in\mathcal I}\|C_i\|<\infty.
\]
For a length-\(k\) product, direct multiplication gives
\[
M_{i_{k-1}}\cdots M_{i_0}
=
\begin{bmatrix}
B_{i_{k-1}}\cdots B_{i_0} & E_k\\
0 & D_{i_{k-1}}\cdots D_{i_0}
\end{bmatrix},
\]
where, with empty products interpreted as identity matrices,
\[
E_k=
\sum_{j=0}^{k-1}
B_{i_{k-1}}\cdots B_{i_{j+1}}
C_{i_j}
D_{i_{j-1}}\cdots D_{i_0}.
\]
Therefore,
\begin{align*}
\|E_k\|
&\le
\sum_{j=0}^{k-1}
\left\|B_{i_{k-1}}\cdots B_{i_{j+1}}\right\|
\|C_{i_j}\|
\left\|D_{i_{j-1}}\cdots D_{i_0}\right\| \\
&\le
\sum_{j=0}^{k-1}
K_B q^{k-1-j} K_C K_D q^j \\
&=
K_BK_CK_D\, k\, q^{k-1}.
\end{align*}
The diagonal blocks also satisfy
\[
\left\|B_{i_{k-1}}\cdots B_{i_0}\right\|\le K_Bq^k,
\qquad
\left\|D_{i_{k-1}}\cdots D_{i_0}\right\|\le K_Dq^k.
\]
Hence, for a constant \(K_q>0\) independent of the switching word and of \(k\),
\[
\left\|M_{i_{k-1}}\cdots M_{i_0}\right\|
\le
K_q (k+1)q^k.
\]
Taking the supremum over all length-\(k\) switching words yields
\[
\sup_{i_0,\ldots,i_{k-1}}
\left\|M_{i_{k-1}}\cdots M_{i_0}\right\|^{1/k}
\le
\bigl(K_q(k+1)\bigr)^{1/k}q.
\]
Letting \(k\to\infty\) gives \(\rho(\mathcal M)\le q\). Since this holds for
every \(q>\max\{\rho(\mathcal B),\rho(\mathcal D)\}\), letting
\(q\downarrow\max\{\rho(\mathcal B),\rho(\mathcal D)\}\) gives
\[
\rho(\mathcal M)\le\max\{\rho(\mathcal B),\rho(\mathcal D)\}.
\]
Combining the lower and upper bounds proves the claim.
\end{proof}

\section{Proof of the full Q-VI switching JSR}
\label{app:proof-qvi-full-jsr}

\noindent\textbf{Restatement of \cref{prop:qvi-full-jsr}.}
The deterministic switching family \(\calH=\{A_\pi:\pi\in\Theta\}\) satisfies
\[
\rho(\calH)=\gamma.
\]

\begin{proof}
For every deterministic policy \(\pi\), the matrix \(P\Pi^\pi\) is
row-stochastic. Hence \(A_\pi\one=\gamma\one\) and
\(\|A_\pi\|_\infty=\gamma\). Therefore, for every switching word
\(\pi_0,\ldots,\pi_{k-1}\),
\[
\|A_{\pi_{k-1}}\cdots A_{\pi_0}\|_\infty\le \gamma^k,
\]
which implies \(\rho(\calH)\le\gamma\).

For the reverse inequality, the same induced infinity norm gives
\[
A_{\pi_{k-1}}\cdots A_{\pi_0}\one=\gamma^k\one
\]
for every switching word. Thus
\[
\|A_{\pi_{k-1}}\cdots A_{\pi_0}\|_\infty
\ge
\frac{\|A_{\pi_{k-1}}\cdots A_{\pi_0}\one\|_\infty}{\|\one\|_\infty}
=
\gamma^k .
\]
Taking the supremum over switching words, the \(k\)-th root, and the limit in
\cref{def:jsr} yields \(\rho(\calH)\ge\gamma\). Hence
\(\rho(\calH)=\gamma\).
\end{proof}

\section{Proof of the exact stochastic-policy error representation}
\label{app:proof-exact-representation}

\noindent\textbf{Restatement of \cref{lem:exact-representation}.}
For every \(Q\in\R^{|\calS||\calA|}\), there exists a stochastic policy
\(\mu_Q:\calS\to\Delta_{|\calA|}\) such that
\[
F(Q)-Q^\star
=
A_{\mu_Q}(Q-Q^\star).
\]
Consequently, with \(e:=Q-Q^\star\),
\[
F(Q)-Q=(A_{\mu_Q}-I)e.
\]
In particular, the statement applies to standard Q-VI iterates and to the
deflated Q-VI iterates studied below.

\begin{proof}
Let \(e(s,a):=Q(s,a)-Q^\star(s,a)\). For each state \(s\), let us define
\[
\delta_s:=\max_{a\in\calA}Q(s,a)-\max_{a\in\calA}Q^\star(s,a).
\]
It is first shown that \(\delta_s\) belongs to the convex hull of
\(\{e(s,a):a\in\calA\}\). To this end, choose
\(i\in\Argmax_{a\in\calA} Q(s,a)\) and \(j\in\Argmax_{a\in\calA} Q^\star(s,a)\). Since
\(Q(s,j)\le Q(s,i)\) and \(Q^\star(s,i)\le Q^\star(s,j)\),
\[
e(s,j)
=Q(s,j)-Q^\star(s,j)
\le
Q(s,i)-Q^\star(s,j)
=
\delta_s
\]
and
\[
\delta_s
=
Q(s,i)-Q^\star(s,j)
\le
Q(s,i)-Q^\star(s,i)
=
e(s,i).
\]
Thus \(\delta_s\) lies between two action-error values and can be written as a
convex combination of action errors at state \(s\). Hence there exists
\(\mu_Q(\cdot\mid s)\in\Delta_{|\calA|}\) such that
\[
\delta_s=\sum_{a\in\calA}\mu_Q(a\mid s)e(s,a).
\]
Constructing this stochastic vector independently for each state gives a
stochastic policy \(\mu_Q\) with
\[
\max_{a\in\calA}Q(s,a)-\max_{a\in\calA} Q^\star(s,a)=(\Pi^{\mu_Q}e)(s),
\qquad \forall s\in\calS.
\]
Since \(Q^\star=F(Q^\star)\), we have
\begin{align*}
F(Q)-Q^\star
&=F(Q)-F(Q^\star)\\
&=\gamma P\Pi^{\mu_Q}(Q-Q^\star)\\
&=\gamma P\Pi^{\mu_Q}e\\
&=A_{\mu_Q}e.
\end{align*}
Subtracting \(Q=Q^\star+e\) from both sides gives
\(F(Q)-Q=(A_{\mu_Q}-I)e\).
\end{proof}

\section{Proof of the empirical orthogonal error decomposition}
\label{app:proof-empirical-orthogonal-error-decomposition}

\noindent\textbf{Restatement of \cref{lem:empirical-orthogonal-error-decomposition}.}
For any error vector
\(e_k\in\R^{|\calS||\calA|}\), define \(z_k\) and \(a_k\) as in
\cref{eq:empirical-error-decomposition}:
\[
z_k:=\proj e_k,
\qquad
a_k:=\frac1{|\calS||\calA|}\one^\top e_k.
\]
Then \(z_k\in \operatorname{span}(\one)^\perp\), \(a_k\) is the empirical mean of \(e_k\), and
\[
e_k=a_k\one+z_k = a_k\one+\proj e_k.
\]
Moreover, this decomposition is unique: if
\(e_k=\alpha\one+w\) with \(w\in \operatorname{span}(\one)^\perp\), then
\(\alpha=a_k\) and \(w=z_k\).

\begin{proof}
By the definition of \(\proj\),
\[
z_k
=
\left(I-\frac1{|\calS||\calA|}\one\one^\top\right)e_k
=
e_k-\left(\frac1{|\calS||\calA|}\one^\top e_k\right)\one
=
e_k-a_k\one .
\]
Therefore \(e_k=a_k\one+z_k\). Also,
\[
\one^\top z_k
=
\one^\top e_k
-
a_k\,\one^\top\one
=
\one^\top e_k
-
\frac1{|\calS||\calA|}\one^\top e_k \cdot |\calS||\calA|
=
0,
\]
so \(z_k\in \operatorname{span}(\one)^\perp\). Finally, if \(e_k=\alpha\one+w\) with
\(w\in \operatorname{span}(\one)^\perp\), then applying \((|\calS||\calA|)^{-1}\one^\top\) gives
\(\alpha=(|\calS||\calA|)^{-1}\one^\top e_k=a_k\), and hence
\(w=e_k-a_k\one=z_k\). This proves the decomposition used in
\cref{eq:empirical-error-decomposition}.
\end{proof}

\section{Proof of the projected error recursion for standard Q-VI}
\label{app:proof-ordinary-projected-error-recursion}

\noindent\textbf{Restatement of \cref{lem:ordinary-projected-error-recursion}.}
The standard Q-VI recursion satisfies
\[
z_{k+1}=\bar A_{\mu_k}z_k,
\qquad
k\in\{0,1,2,\ldots\}.
\]

\begin{proof}
Standard Q-VI gives
\[
e_{k+1}=F(Q_k)-Q^\star=A_{\mu_k}e_k,
\qquad
k\in\{0,1,2,\ldots\}.
\]
by \cref{lem:exact-representation}. Applying \(\proj\) to both sides and using
\cref{eq:empirical-error-decomposition,eq:common-ones-eigenrelation,eq:projection-kills-one}
gives
\begin{align*}
z_{k+1}
&=\proj e_{k+1} \\
&=\proj A_{\mu_k}(a_k\one+z_k) \\
&=a_k\proj A_{\mu_k}\one+\proj A_{\mu_k}z_k \\
&=a_k\gamma\proj\one+\proj A_{\mu_k}z_k \\
&=\proj A_{\mu_k}z_k.
\end{align*}
Since \(z_k=\proj e_k\), the vector \(z_k\) already lies in
\(\operatorname{span}(\one)^\perp\), and hence \(\proj z_k=z_k\). Therefore
\[
\proj A_{\mu_k}z_k
=
\proj A_{\mu_k}\proj z_k
=
\bar A_{\mu_k}z_k.
\]
Thus the averaged all-ones component \(a_k\one\) has no effect on the projected
error recursion; it is killed by the output projection because
\(A_{\mu_k}\one=\gamma\one\) remains in the all-ones direction.
\end{proof}

\section{Proof of the projected JSR bound}
\label{app:proof-restricted-jsr-bound}

\noindent\textbf{Restatement of \cref{lem:restricted-jsr-bound}.}
The projected JSR satisfies
\[
\bar\rho\le\gamma.
\]

\begin{proof}
We follow the projected-product argument used in
Lee~\cite[Lemma~8]{lee2026beyond}. For every deterministic policy
\(\pi\), \(A_\pi\one=\gamma\one\). Hence
\[
\proj A_\pi(I-\proj)=0,
\qquad
\proj A_\pi\proj=\proj A_\pi .
\]
Therefore, for every switching word
\(\pi_0,\ldots,\pi_{k-1}\), the projected product satisfies
\[
\bar A_{\pi_{k-1}}\cdots \bar A_{\pi_0}
=
\proj A_{\pi_{k-1}}\cdots A_{\pi_0}\proj .
\]
Indeed, the identity is immediate for \(k=1\). If it holds for a product of
length \(k\), then
\begin{align*}
\bar A_{\pi_k}\bar A_{\pi_{k-1}}\cdots \bar A_{\pi_0}
&=
\proj A_{\pi_k}\proj\,
\proj A_{\pi_{k-1}}\cdots A_{\pi_0}\proj  \\
&=
\proj A_{\pi_k}\proj A_{\pi_{k-1}}\cdots A_{\pi_0}\proj  \\
&=
\proj A_{\pi_k} A_{\pi_{k-1}}\cdots A_{\pi_0}\proj ,
\end{align*}
where the last equality uses \(\proj A_{\pi_k}\proj=\proj A_{\pi_k}\). This
proves the product identity by induction.

Now use the induced infinity norm. Since \(P\Pi^\pi\) is row-stochastic,
\(\|A_\pi\|_\infty=\gamma\) for every deterministic policy \(\pi\). Thus
\[
\|A_{\pi_{k-1}}\cdots A_{\pi_0}\|_\infty
\le
\gamma^k .
\]
Combining this bound with the product identity gives
\[
\|\bar A_{\pi_{k-1}}\cdots \bar A_{\pi_0}\|_\infty
\le
\|\proj\|_\infty^2\,
\gamma^k .
\]
Taking the maximum over all deterministic switching words, taking the \(k\)-th
root, and then letting \(k\to\infty\), the constant
\(\|\proj\|_\infty^2\) disappears. Since the JSR is independent of the chosen
submultiplicative norm,
\[
\bar\rho\le\gamma .
\]
\end{proof}

\section{Proof of sharpness of the classical rate for standard Q-VI}
\label{app:proof-ordinary-sharp}

\noindent\textbf{Restatement of \cref{prop:ordinary-sharp}.}
For standard Q-VI, no uniform global convergence estimate to \(Q^\star\) can
have a rate smaller than \(\gamma\). In particular, if
\[
Q_0=Q^\star+c\one,
\qquad c\neq0,
\]
then standard Q-VI satisfies
\[
Q_k-Q^\star=\gamma^k c\one,
\qquad
\forall k\ge0.
\]

\begin{proof}
The claim follows from the shift identity
\(F(Q^\star+c\one)=Q^\star+\gamma c\one\). Iterating gives
\(Q_k-Q^\star=\gamma^k c\one\). If a bound with rate \(r<\gamma\) were valid for
all initial conditions, this trajectory would require
\(\gamma^k\le C r^k\) for all \(k\), which is impossible for finite \(C\).
\end{proof}

\section{Proof of the fixed-point lemma for residual deflation}
\label{app:proof-deflated-fixed-point}

\noindent\textbf{Restatement of \cref{lem:deflated-fixed-point}.}
The deflated map
\[
T_{\rm def}(Q)
:=
F(Q)
+
\frac{\gamma}{1-\gamma}
\left(
\frac1{|\calS||\calA|}\one^\top(F(Q)-Q)
\right)\one
\]
has the unique fixed point \(Q^\star\).

\begin{proof}
Since \(F(Q^\star)=Q^\star\), \(T_{\rm def}(Q^\star)=Q^\star\). Conversely,
suppose that \(Q=T_{\rm def}(Q)\). Let
\[
r:=F(Q)-Q,
\qquad
\bar r:=(|\calS||\calA|)^{-1}\one^\top r .
\]
Then
\[
r=-\frac{\gamma}{1-\gamma}\bar r\one.
\]
Taking the empirical mean gives
\[
\bar r=-\frac{\gamma}{1-\gamma}\bar r.
\]
Hence \(\bar r=0\), and therefore \(r=0\). Thus \(F(Q)=Q\). Since the Bellman
operator has the unique fixed point \(Q^\star\), \(Q=Q^\star\).
\end{proof}

\section{Proof of scalar-shift equivalence to standard Q-VI}
\label{app:proof-scalar-shift-equivalence}

\noindent\textbf{Restatement of \cref{lem:scalar-shift-equivalence}.}
Let \(U_{k+1}=F(U_k)\) be standard Q-VI with \(U_0=Q_0\), and let
\(Q_k\) be generated by \cref{eq:deflated-vi} from the same initialization. Then
there exist scalars \(c_k\in\R\) such that
\[
Q_k=U_k+c_k\one,
\qquad \forall k\ge0.
\]
Consequently,
\[
\proj Q_k=\proj U_k,
\qquad
\pi_{Q_k}=\pi_{U_k},
\qquad \forall k\ge0,
\]
under the same tie-breaking rule.

\begin{proof}
The claim holds at \(k=0\) with \(c_0=0\). Suppose that
\(Q_k=U_k+c_k\one\). By the shift identity \cref{eq:shift-identity},
\[
F(Q_k)=F(U_k)+\gamma c_k\one=U_{k+1}+\gamma c_k\one.
\]
Moreover,
\[
F(Q_k)-Q_k=F(U_k)-U_k+(\gamma-1)c_k\one.
\]
Taking the empirical mean and substituting into \cref{eq:deflated-vi} yields
\begin{align*}
Q_{k+1}
&=
U_{k+1}+\gamma c_k\one
+
\frac{\gamma}{1-\gamma}
\left(
\frac1{|\calS||\calA|}\one^\top(F(U_k)-U_k)+(\gamma-1)c_k
\right)\one \\
&=
U_{k+1}
+
\frac{\gamma}{1-\gamma}
\left(
\frac1{|\calS||\calA|}\one^\top(F(U_k)-U_k)
\right)\one .
\end{align*}
Thus \(Q_{k+1}=U_{k+1}+c_{k+1}\one\) for a scalar \(c_{k+1}\). The induction is
complete. Projection by \(\proj\) eliminates the scalar shift, and adding a
state-independent constant to all state-action values does not change any
statewise greedy action under the same tie-breaking rule.
\end{proof}

\section{Proof of the orthogonal decomposition of the Q-error}
\label{app:proof-error-decomposition}

\noindent\textbf{Restatement of \cref{lem:error-decomposition}.}
Let \(Q_k\) be generated by~\cref{eq:deflated-vi}, and define
\[
e_k:=Q_k-Q^\star,
\qquad
z_k:=\proj e_k .
\]
Then \(z_k\in \operatorname{span}(\one)^\perp\). Moreover, if
\[
a_k:=\frac1{|\calS||\calA|}\one^\top e_k,
\]
then the full error admits the orthogonal decomposition
\[
e_k
=
Q_k-Q^\star
=
a_k\one+\proj e_k
=
a_k\one+z_k .
\]
Here \(a_k\one\in\operatorname{span}(\one)\) is the all-ones component of the
error, while \(z_k\in\operatorname{span}(\one)^\perp\) is its orthogonal
transverse component. In addition, this decomposition is unique.

\begin{proof}
By the definition of the projection,
\[
\proj
=
I-\frac1{|\calS||\calA|}\one\one^\top .
\]
Therefore,
\[
z_k
=
\proj e_k
=
\left(I-\frac1{|\calS||\calA|}\one\one^\top\right)e_k
=
e_k-\left(\frac1{|\calS||\calA|}\one^\top e_k\right)\one
=
e_k-a_k\one .
\]
Rearranging gives
\[
e_k=a_k\one+z_k=a_k\one+\proj e_k.
\]

It remains to verify that \(z_k\) is orthogonal to the all-ones direction. Since
\(\one^\top\one=|\calS||\calA|\), we have
\[
\one^\top z_k
=
\one^\top e_k
-
a_k\one^\top\one
=
\one^\top e_k
-
\frac1{|\calS||\calA|}\one^\top e_k\cdot |\calS||\calA|
=
0 .
\]
Hence \(z_k\in \operatorname{span}(\one)^\perp\).

Finally, suppose that another decomposition is given by
\[
e_k=\alpha\one+w,
\qquad
w\in \operatorname{span}(\one)^\perp .
\]
Multiplying both sides by \((|\calS||\calA|)^{-1}\one^\top\) yields
\[
\frac1{|\calS||\calA|}\one^\top e_k
=
\alpha+\frac1{|\calS||\calA|}\one^\top w
=
\alpha,
\]
because \(\one^\top w=0\). Thus \(\alpha=a_k\), and consequently
\[
w=e_k-a_k\one=z_k .
\]
Therefore, the decomposition is unique.
\end{proof}

\section{Proof of the exact deflated switching dynamics}
\label{app:proof-exact-deflated-dynamics}

\noindent\textbf{Restatement of \cref{prop:exact-deflated-dynamics}.}
Let \(Q_k\) be generated by~\cref{eq:deflated-vi}, and define
\[
e_k:=Q_k-Q^\star,
\qquad
z_k:=\proj e_k,
\qquad
a_k:=\frac1{|\calS||\calA|}\one^\top e_k .
\]
Then
\[
z_{k+1}=\bar A_{\mu_k}z_k,
\qquad
k\in\{0,1,2,\ldots\},
\]
and
\[
a_{k+1}=\ell_{\mu_k}z_k,
\qquad
\ell_\mu:=\frac{1}{|\calS||\calA|(1-\gamma)}\one^\top A_\mu,
\qquad
k\in\{0,1,2,\ldots\}.
\]
Equivalently,
\[
\begin{bmatrix}
a_{k+1}\\ z_{k+1}
\end{bmatrix}
=
D_{\mu_k}
\begin{bmatrix}
a_k\\ z_k
\end{bmatrix},
\qquad
D_\mu:=
\begin{bmatrix}
0 & \ell_\mu\\
0 & \bar A_\mu
\end{bmatrix},
\qquad
k\in\{0,1,2,\ldots\}.
\]

\begin{proof}
By \cref{lem:exact-representation}, the Bellman residual \(r_k\) satisfies
\[
r_k=F(Q_k)-Q_k=(A_{\mu_k}-I)e_k.
\]
Subtracting \(Q^\star\) from \cref{eq:deflated-vi} and using
\(F(Q_k)-Q^\star=A_{\mu_k}e_k\) gives
\begin{equation}
\label{eq:deflated-error-recursion}
e_{k+1}
=
A_{\mu_k}e_k
+
\frac{\gamma}{1-\gamma}\bar r_k\one,
\qquad
k\in\{0,1,2,\ldots\}.
\end{equation}
Applying \(\proj\) to \cref{eq:deflated-error-recursion} and using
\cref{eq:deflated-error-decomposition,eq:common-ones-eigenrelation,eq:projection-kills-one}
gives
\[
z_{k+1}
=
\proj A_{\mu_k}z_k
=
\proj A_{\mu_k}\proj z_k
=
\bar A_{\mu_k}z_k.
\]

It remains to compute the scalar component. The residual mean is obtained from
\(r_k=(A_{\mu_k}-I)e_k\) by applying the empirical averaging functional
\((|\calS||\calA|)^{-1}\one^\top\). Substituting
\(e_k=a_k\one+z_k\) from \cref{eq:deflated-error-decomposition} gives
\begin{align*}
\bar r_k
&=
\frac1{|\calS||\calA|}\one^\top(A_{\mu_k}-I)e_k \\
&=
\frac1{|\calS||\calA|}\one^\top(A_{\mu_k}-I)(a_k\one+z_k) \\
&=
\frac{a_k}{|\calS||\calA|}\one^\top(\gamma\one-\one)
+
\frac1{|\calS||\calA|}\one^\top A_{\mu_k}z_k
-
\frac1{|\calS||\calA|}\one^\top z_k .
\end{align*}
Here \(A_{\mu_k}\one=\gamma\one\) by
\cref{eq:common-ones-eigenrelation}, and \(\one^\top z_k=0\) because
\(z_k=\proj e_k\in \operatorname{span}(\one)^\perp\). Therefore
\[
\bar r_k
=
(\gamma-1)a_k+
\frac1{|\calS||\calA|}\one^\top A_{\mu_k}z_k.
\]
Taking empirical means in \cref{eq:deflated-error-recursion} gives
\begin{align*}
a_{k+1}
&=
\gamma a_k+\frac1{|\calS||\calA|}\one^\top A_{\mu_k}z_k
+
\frac{\gamma}{1-\gamma}
\left((\gamma-1)a_k+\frac1{|\calS||\calA|}\one^\top A_{\mu_k}z_k\right)\\
&=
\frac{1}{|\calS||\calA|(1-\gamma)}\one^\top A_{\mu_k}z_k.
\end{align*}
This proves the stated block representation.
\end{proof}

\section{Proof of the associated full switched-family JSR theorem}
\label{app:proof-deflated-jsr}

\noindent\textbf{Restatement of \cref{thm:deflated-jsr}.}
The associated full stochastic-policy switched family of the deflated Q-VI
iteration \cref{eq:deflated-vi} satisfies
\[
\rho_{\rm def}=\bar\rho.
\]
Consequently, every actual deflated Q-VI trajectory has an upper asymptotic
rate no larger than \(\bar\rho\) in the switched-family sense made explicit in
\cref{thm:explicit-deflated-rate}. Since \(\bar\rho\le\gamma\), this certified
rate is no worse than the classical worst-case full Q-VI rate identified in
\cref{prop:qvi-full-jsr,prop:ordinary-sharp}. If \(\bar\rho<\gamma\), the
certified switched-family rate is strictly smaller than \(\gamma\).

\begin{proof}
By \cref{prop:exact-deflated-dynamics}, the associated deflated Q-VI family is
precisely \(\mathcal D_{\rm def}\). Thus, by the definition of
\(\rho_{\rm def}\) and by \cref{lem:block-upper-triangular-jsr},
\[
\rho_{\rm def}
=
\max\left\{0,
\rho\left(\{\bar A_\mu:\mu\text{ stochastic}\}\right)
\right\}
=
\rho\left(\{\bar A_\mu:\mu\text{ stochastic}\}\right).
\]
The stochastic-policy family is the convex hull of the deterministic-policy
family, because every stochastic stationary policy can be written as a convex
combination of deterministic stationary policies. Since \(\mu\mapsto A_\mu\) is
affine and \(\bar A_\mu=\proj A_\mu\proj\) is linear in \(A_\mu\),
\[
\{\bar A_\mu:\mu\text{ stochastic}\}
=
\co\left(\{\bar A_\pi:\pi\in\Theta\}\right).
\]
By \cref{lem:jsr-convexification},
\[
\rho\left(\{\bar A_\mu:\mu\text{ stochastic}\}\right)
=
\rho\left(\{\bar A_\pi:\pi\in\Theta\}\right)
=
\bar\rho.
\]
Therefore \(\rho_{\rm def}=\bar\rho\). Since \(\bar\rho\le\gamma\), the rate
comparison follows.
\end{proof}

\section{Proof of the explicit full convergence rate theorem}
\label{app:proof-explicit-deflated-rate}

\noindent\textbf{Restatement of \cref{thm:explicit-deflated-rate}.}
Consider \cref{eq:deflated-vi}. Fix any
\(\varepsilon\in(0,1-\bar\rho)\), and define
\[
\beta_\varepsilon:=\bar\rho+\varepsilon .
\]
Then there exists \(C_{\beta_\varepsilon}>0\) such that, for all \(k\ge1\),
\[
\|Q_k-Q^\star\|_2
\le
C_{\beta_\varepsilon}\beta_\varepsilon^{k-1}\|\proj(Q_0-Q^\star)\|_2.
\]
Consequently,
\[
\limsup_{k\to\infty}\|Q_k-Q^\star\|_2^{1/k}\le\bar\rho.
\]
If \(\bar\rho<\gamma\), then deflated Q-VI has a certified full-error convergence rate strictly smaller than the classical worst-case full-error rate of standard Q-VI.

\begin{proof}
By~\cref{thm:deflated-jsr}, the projected stochastic-policy family has the same
JSR as the deterministic projected family, namely \(\bar\rho\). Fix
\(\varepsilon\in(0,1-\bar\rho)\) and define
\(\beta_\varepsilon:=\bar\rho+\varepsilon\). Applying
\cref{lem:jsr-product-bound} to the projected stochastic-policy family gives a
constant \(c_{\beta_\varepsilon}>0\) such that, for every length \(m\ge0\) and
every stochastic-policy switching word,
\[
\|\bar A_{\mu_{m-1}}\cdots\bar A_{\mu_0}\|_2
\le
c_{\beta_\varepsilon}\beta_\varepsilon^m.
\]

Iterating the projected recursion in \cref{prop:exact-deflated-dynamics} gives
\[
z_k
=
\bar A_{\mu_{k-1}}\cdots\bar A_{\mu_0}z_0.
\]
Therefore
\[
\|z_k\|_2
\le
\|\bar A_{\mu_{k-1}}\cdots\bar A_{\mu_0}\|_2\|z_0\|_2
\le
c_{\beta_\varepsilon}\beta_\varepsilon^k\|z_0\|_2.
\]

Next, recall that \(\ell_\mu\) is the row vector defined in
\cref{prop:exact-deflated-dynamics},
\[
\ell_\mu:=\frac{1}{|\calS||\calA|(1-\gamma)}\one^\top A_\mu.
\]
It is the linear functional that maps the previous projected error \(z_k\) to
the next averaged scalar component \(a_{k+1}\). Since the set of stochastic
policies is a finite product of compact simplexes and \(\mu\mapsto\ell_\mu\) is
continuous, its Euclidean operator norm is uniformly bounded. Hence
\[
L_\ell:=\sup_{\mu}\|\ell_\mu\|_2<\infty.
\]
For \(k\ge1\), \cref{prop:exact-deflated-dynamics} gives
\[
|a_k|=|\ell_{\mu_{k-1}}z_{k-1}|
\le L_\ell\|z_{k-1}\|_2
\le
L_\ell c_{\beta_\varepsilon}\beta_\varepsilon^{k-1}\|z_0\|_2.
\]
Since \(e_k=a_k\one+z_k\),
\begin{align*}
\|e_k\|_2
&\le
\|a_k\one\|_2+\|z_k\|_2\\
&=
\sqrt{|\calS||\calA|}|a_k|+\|z_k\|_2\\
&\le
c_{\beta_\varepsilon}(\sqrt{|\calS||\calA|} L_\ell+\beta_\varepsilon)
\beta_\varepsilon^{k-1}\|z_0\|_2.
\end{align*}
Taking
\[
C_{\beta_\varepsilon}:=
c_{\beta_\varepsilon}(\sqrt{|\calS||\calA|} L_\ell+\beta_\varepsilon)
\]
and using \(z_0=\proj(Q_0-Q^\star)\) proves the finite-time bound.

It remains to prove the limsup conclusion. If \(z_0=0\), then
\cref{prop:exact-deflated-dynamics} gives \(z_k=0\) and \(a_k=0\) for all
\(k\ge1\), so \(e_k=0\) for all \(k\ge1\) and the claim is immediate. Otherwise,
the finite-time bound implies
\[
\|e_k\|_2^{1/k}
\le
\left(C_{\beta_\varepsilon}\|z_0\|_2\beta_\varepsilon^{-1}\right)^{1/k}
\beta_\varepsilon.
\]
The constant factor satisfies
\[
\lim_{k\to\infty}
\left(C_{\beta_\varepsilon}\|z_0\|_2\beta_\varepsilon^{-1}\right)^{1/k}=1.
\]
Therefore
\[
\limsup_{k\to\infty}\|Q_k-Q^\star\|_2^{1/k}
=
\limsup_{k\to\infty}\|e_k\|_2^{1/k}
\le \beta_\varepsilon.
\]
Since this holds for every \(\varepsilon\in(0,1-\bar\rho)\), letting
\(\varepsilon\downarrow0\) gives
\[
\limsup_{k\to\infty}\|Q_k-Q^\star\|_2^{1/k}\le\bar\rho.
\]
\end{proof}

\section{Proof of the scalar lifting bound}
\label{app:proof-scalar-lifting-bound}

\noindent\textbf{Restatement of \cref{lem:scalar-lifting-bound}.}
For
\[
\ell_\mu:=\frac{1}{|\calS||\calA|(1-\gamma)}\one^\top A_\mu,
\]
one has
\[
L_\ell:=\sup_\mu\|\ell_\mu\|_2\le \frac{\gamma}{1-\gamma}.
\]
Consequently, if the projected stochastic-policy family satisfies
\[
\|\bar A_{\mu_{m-1}}\cdots\bar A_{\mu_0}\|_2
\le c_{\beta_\varepsilon}\beta_\varepsilon^m,
\qquad
\beta_\varepsilon:=\bar\rho+\varepsilon,
\]
for all \(m\ge0\) and all stochastic-policy switching words, then the constant
in the full-error estimate can be chosen as
\[
C_{\beta_\varepsilon}
:=c_{\beta_\varepsilon}
\left(\frac{\sqrt{|\calS||\calA|}\,\gamma}{1-\gamma}+\beta_\varepsilon\right).
\]
Moreover, if \(N_{\beta_\varepsilon}\) is the finite threshold in
\cref{lem:jsr-product-bound} for this projected family, then one may use the
coarse bound
\[
 c_{\beta_\varepsilon}
 \le
 \max\left\{1,\max_{0\le m<N_{\beta_\varepsilon}}
 \left(\frac{\gamma\sqrt{|\calS||\calA|}}{\beta_\varepsilon}\right)^m\right\},
\]
and hence
\[
C_{\beta_\varepsilon}
\le
\max\left\{1,\max_{0\le m<N_{\beta_\varepsilon}}
\left(\frac{\gamma\sqrt{|\calS||\calA|}}{\beta_\varepsilon}\right)^m\right\}
\left(\frac{\sqrt{|\calS||\calA|}\,\gamma}{1-\gamma}+\beta_\varepsilon\right).
\]

\begin{proof}
For every stochastic policy \(\mu\), the matrix \(P\Pi^\mu\) is row-stochastic.
Hence \(A_\mu=\gamma P\Pi^\mu\) is nonnegative and every row of \(A_\mu\) sums to
\(\gamma\). Therefore the nonnegative row vector \(\one^\top A_\mu\) has total
mass
\[
\|\one^\top A_\mu\|_1=\one^\top A_\mu\one=\gamma |\calS||\calA|.
\]
Since \(\|v\|_2\le\|v\|_1\) for every vector \(v\),
\[
\|\ell_\mu\|_2
=\frac{1}{|\calS||\calA|(1-\gamma)}\|\one^\top A_\mu\|_2
\le\frac{1}{|\calS||\calA|(1-\gamma)}\|\one^\top A_\mu\|_1
=\frac{\gamma}{1-\gamma}.
\]
Taking the supremum over \(\mu\) gives the first claim. The stated value of
\(C_{\beta_\varepsilon}\) follows by substituting this bound on \(L_\ell\) into
the estimate in \cref{thm:explicit-deflated-rate}. For the coarse bound on
\(c_{\beta_\varepsilon}\), note that
\(\|\bar A_\mu\|_2\le\|A_\mu\|_2\le\gamma\sqrt{|\calS||\calA|}\), because
\(\|P\Pi^\mu\|_\infty=1\), \(\|P\Pi^\mu\|_1\le |\calS||\calA|\), and
\(\|M\|_2\le\sqrt{\|M\|_1\|M\|_\infty}\). Therefore every length-\(m\) projected
product has Euclidean norm at most \((\gamma\sqrt{|\calS||\calA|})^m\). Combining this
finite-length bound with the construction of \(c_{\beta_\varepsilon}\) in
\cref{lem:jsr-product-bound} gives the displayed upper bounds.
\end{proof}

\section{Proof of the pure all-ones error corollary}
\label{app:proof-pure-ones}

\noindent\textbf{Restatement of \cref{cor:pure-ones}.}
If \(Q_0-Q^\star\in\operatorname{span}(\one)\), then deflated Q-VI satisfies
\[
Q_1=Q^\star.
\]

\begin{proof}
If \(Q_0-Q^\star\in\operatorname{span}(\one)\), then \(z_0=0\). By
\cref{prop:exact-deflated-dynamics}, \(z_1=0\) and
\(a_1=\ell_{\mu_0}z_0=0\). Thus \(Q_1-Q^\star=0\).
\end{proof}

\section{Proof of the iteration complexity corollary}
\label{app:proof-iteration-complexity}

\noindent\textbf{Restatement of \cref{cor:iteration-complexity}.}
Let \(\varepsilon_Q>0\), fix \(\varepsilon\in(0,1-\bar\rho)\), and define
\(\beta_\varepsilon:=\bar\rho+\varepsilon\). If
\(\proj(Q_0-Q^\star)=0\), then \(Q_1=Q^\star\). Otherwise, the bound
\[
k
\ge
1+
\left\lceil
\frac{
\ln\left(C_{\beta_\varepsilon}\|\proj(Q_0-Q^\star)\|_2/\varepsilon_Q\right)
}{
-\ln \beta_\varepsilon
}
\right\rceil
\]
implies
\[
\|Q_k-Q^\star\|_2\le\varepsilon_Q.
\]

\begin{proof}
If \(\proj(Q_0-Q^\star)=0\), then \cref{cor:pure-ones} gives
\(Q_1=Q^\star\). Otherwise, \cref{thm:explicit-deflated-rate} gives
\[
\|Q_k-Q^\star\|_2
\le
C_{\beta_\varepsilon}\beta_\varepsilon^{k-1}
\|\proj(Q_0-Q^\star)\|_2.
\]
Since \(0<\beta_\varepsilon<1\), the stated lower bound on \(k\) implies
\[
C_{\beta_\varepsilon}\beta_\varepsilon^{k-1}
\|\proj(Q_0-Q^\star)\|_2
\le
\varepsilon_Q.
\]
Substituting this inequality into the explicit rate bound proves the claim.
\end{proof}

\section{Proof of the fixed-point lemma for the \(d\)-weighted map}
\label{app:proof-d-weighted-fixed-point}

\noindent\textbf{Restatement of \cref{lem:d-weighted-fixed-point}.}
The map \(T_d\) in \cref{eq:d-weighted-deflated-map} has the unique fixed point
\(Q^\star\).

\begin{proof}
Since \(F(Q^\star)=Q^\star\), \(T_d(Q^\star)=Q^\star\). Conversely, suppose that
\(Q=T_d(Q)\). Let
\[
r:=F(Q)-Q,
\qquad
\bar r_d:=d^\top r.
\]
Then
\[
r=-\frac{\gamma}{1-\gamma}\bar r_d\one.
\]
Multiplying by \(d^\top\) and using \(d^\top\one=1\) gives
\[
\bar r_d=-\frac{\gamma}{1-\gamma}\bar r_d.
\]
Hence \(\bar r_d=0\), and therefore \(r=0\). Thus \(F(Q)=Q\). Since the Bellman
optimality operator has the unique fixed point \(Q^\star\), \(Q=Q^\star\).
\end{proof}

\section{Proof of the exact \(d\)-projected switching dynamics}
\label{app:proof-d-weighted-dynamics}

\noindent\textbf{Restatement of \cref{prop:d-weighted-dynamics}.}
Let \(Q_k\) be generated by \cref{eq:d-weighted-deflated-vi}, and let
\(e_k=Q_k-Q^\star\). For the stochastic-policy representation in
\cref{lem:exact-representation}, define
\[
z_k:=\boldsymbol{\Pi}_d e_k,
\qquad
 a_k:=d^\top e_k.
\]
Then
\begin{align*}
 z_{k+1}
 &=
 \boldsymbol{\Pi}_d A_{\mu_k}\boldsymbol{\Pi}_d\, z_k,
 \qquad
 k\in\{0,1,2,\ldots\},
 \\
 a_{k+1}
 &=
 \ell_{\mu_k}z_k,
 \qquad
 \ell_\mu:=\frac{1}{1-\gamma}d^\top A_\mu,
 \qquad
 k\in\{0,1,2,\ldots\}.
\end{align*}
Equivalently,
\[
\begin{bmatrix}
 a_{k+1}\\ z_{k+1}
\end{bmatrix}
=
D_{\mu_k}
\begin{bmatrix}
 a_k\\ z_k
\end{bmatrix},
\qquad
D_{\mu}:=
\begin{bmatrix}
0 & \ell_\mu\\
0 & \boldsymbol{\Pi}_d A_\mu\boldsymbol{\Pi}_d
\end{bmatrix},
\qquad
k\in\{0,1,2,\ldots\},
\]
where the lower-right block is regarded as a linear map from \(W_d\) to
\(W_d\). Thus the all-ones component again has zero autonomous diagonal
dynamics.

\begin{proof}
By \cref{lem:exact-representation},
\[
F(Q_k)-Q_k=(A_{\mu_k}-I)e_k.
\]
The error recursion induced by \cref{eq:d-weighted-deflated-map} is
\[
e_{k+1}
=
A_{\mu_k}e_k
+
\frac{\gamma}{1-\gamma}
\one d^\top(A_{\mu_k}-I)e_k,
\qquad
k\in\{0,1,2,\ldots\}.
\]
Applying \(\boldsymbol{\Pi}_d\) eliminates the correction term because \(\boldsymbol{\Pi}_d\one=0\). With
\(e_k=a_k\one+z_k\) and \(A_{\mu_k}\one=\gamma\one\),
\[
z_{k+1}
=
\boldsymbol{\Pi}_d A_{\mu_k}z_k
=
\boldsymbol{\Pi}_d A_{\mu_k}\boldsymbol{\Pi}_d z_k.
\]
Since \(z_k\in W_d\) and \(\boldsymbol{\Pi}_d A_{\mu_k}\boldsymbol{\Pi}_d\) maps
\(W_d\) into \(W_d\), this is precisely its action on \(z_k\).

Next, taking the \(d\)-average gives
\begin{align*}
a_{k+1}
&=
d^\top A_{\mu_k}e_k
+
\frac{\gamma}{1-\gamma}d^\top(A_{\mu_k}-I)e_k\\
&=
\frac{1}{1-\gamma}d^\top A_{\mu_k}e_k
-
\frac{\gamma}{1-\gamma}d^\top e_k.
\end{align*}
Substituting \(e_k=a_k\one+z_k\) and using
\(A_{\mu_k}\one=\gamma\one\), \(d^\top\one=1\), and
\(d^\top z_k=0\) yields
\[
a_{k+1}
=
\frac{1}{1-\gamma}d^\top A_{\mu_k}z_k.
\]
This proves the block representation.
\end{proof}

\section{Proof of the distribution-weighted JSR theorem}
\label{app:proof-d-weighted-jsr}

\noindent\textbf{Restatement of \cref{thm:d-weighted-jsr}.}
For any fixed \(d\in\Delta_{|\calS||\calA|}\), the associated full stochastic-policy
switched-family JSR of \cref{eq:d-weighted-deflated-vi} is
\[
\rho_d=\bar\rho\le\gamma.
\]
Consequently, the distribution-weighted deflated Q-VI map has the same
associated switched-family rate as the uniform deflated Q-VI map in
\cref{eq:deflated-vi}.

\begin{proof}
By \cref{prop:d-weighted-dynamics}, the associated switched family is block
upper triangular with diagonal blocks \(0\) and
\(\boldsymbol{\Pi}_d A_\mu\boldsymbol{\Pi}_d\), where the latter block acts on \(W_d\). By
\cref{lem:block-upper-triangular-jsr}, the full switched-family JSR equals the
operator JSR of this lower-right family on \(W_d\). More explicitly, this means
\[
\rho_{W_d}\!\left(\{\boldsymbol{\Pi}_d A_\mu\boldsymbol{\Pi}_d:\mu\text{ stochastic}\}\right)
:=
\lim_{k\to\infty}
\sup_{\mu_0,\ldots,\mu_{k-1}}
\left\|
(\boldsymbol{\Pi}_d A_{\mu_{k-1}}\boldsymbol{\Pi}_d)\cdots
(\boldsymbol{\Pi}_d A_{\mu_0}\boldsymbol{\Pi}_d)
\right\|_{\mathrm{op},d}^{1/k},
\]
where each factor is regarded as a linear map \(W_d\to W_d\), and
\(\|\cdot\|_{\mathrm{op},d}\) is any induced operator norm on \(W_d\). Thus
\[
\rho_d
=
\rho_{W_d}\!\left(\{\boldsymbol{\Pi}_d A_\mu\boldsymbol{\Pi}_d:\mu\text{ stochastic}\}\right).
\]
It remains to compare this value with \(\bar\rho\). Recall that
\[
\operatorname{span}(\one)^\perp=\{x\in\R^{|\calS||\calA|}:\one^\top x=0\},
\qquad
W_d:=\operatorname{span}(d)^\perp
=\{x\in\R^{|\calS||\calA|}:d^\top x=0\}.
\]
The map \(S_d:\operatorname{span}(\one)^\perp\to W_d\), \(S_d x:=\boldsymbol{\Pi}_d x\), is a linear bijection with
inverse \(S_d^{-1}y=\proj y\). Indeed, if \(x\in \operatorname{span}(\one)^\perp\) and
\(\boldsymbol{\Pi}_d x=0\), then \(x=(d^\top x)\one\), and the constraint \(\one^\top x=0\)
implies \(x=0\). For \(y\in W_d\), \(\boldsymbol{\Pi}_d\proj y=\boldsymbol{\Pi}_d y=y\); for
\(x\in \operatorname{span}(\one)^\perp\), \(\proj\boldsymbol{\Pi}_d x=\proj x=x\). This remains true even when some
components of \(d\) are zero; only \(d^\top\one=1\) is needed.

For any deterministic policy \(\pi\) and any \(y\in W_d\),
\begin{align*}
S_d(\proj A_\pi\proj)S_d^{-1}y
&=
\boldsymbol{\Pi}_d\proj A_\pi\proj y \\
&=
\boldsymbol{\Pi}_d A_\pi\proj y \\
&=
\boldsymbol{\Pi}_d A_\pi y \\
&=
\boldsymbol{\Pi}_d A_\pi\boldsymbol{\Pi}_d y .
\end{align*}
The second equality uses \(\boldsymbol{\Pi}_d\proj=\boldsymbol{\Pi}_d\). The third equality follows because
\(y-\proj y\in\operatorname{span}(\one)\) and
\(\boldsymbol{\Pi}_d A_\pi\one=\gamma\boldsymbol{\Pi}_d\one=0\). The last equality uses \(y\in W_d\), so
\(\boldsymbol{\Pi}_d y=y\). Thus the deterministic family
\(\{\boldsymbol{\Pi}_d A_\pi\boldsymbol{\Pi}_d:\pi\in\Theta\}\), regarded on \(W_d\), is similar to the
orthogonally projected family
\(\{\proj A_\pi\proj:\pi\in\Theta\}\), regarded on \(\operatorname{span}(\one)^\perp\). Similarity
preserves the JSR.

The stochastic-policy family is the convex hull of the deterministic-policy
family, and \cref{lem:jsr-convexification} gives the same JSR after
convexification. Therefore, writing \(\rho_{\operatorname{span}(\one)^\perp}\) for the analogous
operator JSR on \(\operatorname{span}(\one)^\perp\),
\[
\rho_d
=
\rho_{W_d}\!\left(\{\boldsymbol{\Pi}_d A_\pi\boldsymbol{\Pi}_d:\pi\in\Theta\}\right)
=
\rho_{\operatorname{span}(\one)^\perp}\!\left(\{\proj A_\pi\proj:\pi\in\Theta\}\right)
=
\bar\rho.
\]
Finally, \(\bar\rho\le\gamma\) by \cref{lem:restricted-jsr-bound}.
\end{proof}

\bibliographystyle{plainnat}
\bibliography{reference}

\end{document}